\documentclass[12pt,a4paper,leqno]{amsart}

\usepackage[latin1]{inputenc}
\usepackage[T1]{fontenc}
\usepackage{amsfonts}
\usepackage{amsmath}
\usepackage{amssymb}
\usepackage{eurosym}
\usepackage{mathrsfs}
\usepackage{palatino}
\usepackage{color}
\usepackage{esint}
\usepackage{url}
\usepackage{verbatim}

\usepackage{enumerate}

\newcommand{\R}{\mathbb{R}}
\newcommand{\C}{\mathbb{C}}

\newcommand{\Q}{\mathbb{Q}}

\newcommand{\Z}{\mathbb{Z}}
\newcommand{\E}{\mathbb{E}}

\newcommand{\calB}{\mathcal{B}}

\newcommand{\calM}{\mathcal{M}}
\newcommand{\calS}{\mathcal{S}}

\newcommand{\calF}{\mathcal{F}}

\newcommand{\calL}{\mathcal{L}}
\newcommand{\calK}{\mathcal{K}}

\newcommand{\calI}{\mathcal{I}}
\newcommand{\calR}{\mathcal{R}}
\newcommand{\calT}{\mathcal{T}}

\newcommand{\scrA}{\mathscr{A}}
\newcommand{\calJ}{\mathcal{J}}

\newcommand{\bla}{\big \langle}
\newcommand{\bra}{\big \rangle}

\numberwithin{equation}{section}

\newcommand{\ud}[0]{\,\mathrm{d}}

\newcommand{\BMO}[0]{\operatorname{BMO}}

\newcommand{\loc}[0]{\operatorname{loc}}

\newcommand{\calD}[0]{\mathcal{D}}

\swapnumbers
\theoremstyle{plain}
\newtheorem{thm}[equation]{Theorem}
\newtheorem{lem}[equation]{Lemma}
\newtheorem{prop}[equation]{Proposition}
\newtheorem{cor}[equation]{Corollary}

\theoremstyle{definition}

\theoremstyle{remark}
\newtheorem{rem}[equation]{Remark}

\title[Multi-parameter estimates via operator-valued shifts]{Multi-parameter estimates via operator-valued shifts}

\author{Tuomas Hyt\"onen}
\address[T.H.]{Department of Mathematics and Statistics, University of Helsinki, P.O.B. 68, FI-00014 University of Helsinki, Finland}
\email{tuomas.hytonen@helsinki.fi}

\author{Henri Martikainen}
\address[H.M.]{Department of Mathematics and Statistics, University of Helsinki, P.O.B. 68, FI-00014 University of Helsinki, Finland}
\email{henri.martikainen@helsinki.fi}

\author{Emil Vuorinen}
\address[E.V.]{Department of Mathematics and Statistics, University of Helsinki, P.O.B. 68, FI-00014 University of Helsinki, Finland}
\email{emil.vuorinen@helsinki.fi}

\makeatletter
\@namedef{subjclassname@2010}{%
  \textup{2010} Mathematics Subject Classification}
\makeatother


\begin{document}

\begin{abstract}
We prove new results for multi-parameter singular integrals. For example, we prove that bi-parameter singular integrals in $\R^{n+m}$
satisfying natural $T1$ type conditions map $L^q(\R^n; L^p(\R^m;E))$ to $L^q(\R^n; L^p(\R^m;E))$ for
all $p,q \in (1,\infty)$ and UMD function lattices $E$. This result is shown to hold even in the $\calR$-boundedness sense
for all suitable families of bi-parameter singular integrals. 
On the technique side we demonstrate how many dyadic multi-parameter operators can be bounded by using, and further developing,
the theory of operator-valued dyadic shifts.
Even in the scalar-valued case this is an efficient way to bound  the various so called partial paraproducts,
which are key operators appearing in the multi-parameter representation theorems.
Our proofs also entail verifying the $\calR$-boundedness of various families of multi-parameter paraproducts.
\end{abstract}

\maketitle
\tableofcontents

\section{Introduction}
Representation theorems show the exact dyadic structure behind Calder\'on--Zygmund operators by representing them using simple dyadic operators, namely
some cancellative dyadic shifts and various paraproducts. In the linear case Petermichl \cite{Pe} first represented the Hilbert transform in this way, and later
one of us \cite{Hy} proved a representation theorem for all linear Calder\'on--Zygmund operators. These are important theorems as they can be used to reduce
questions to dyadic model operators. Such theorems are proved using dyadic--probabilistic methods, which were first pioneered by Nazarov--Treil--Volberg (see e.g. \cite{NTV}).

Dyadic--probabilistic methods of harmonic analysis have recently really shown their power in the multi-parameter setting.
For example, a representation theorem holds also in the bi-parameter setting as shown by one of us \cite{Ma1}. The multi-parameter extension of this is by Y. Ou \cite{Yo}.
In the bi-parameter context the representation theorem has proved to be extremely useful e.g. in connection
with bi-parameter commutators and weighted analysis, see Holmes--Petermichl--Wick \cite{HPW} and Ou--Petermichl--Strouse \cite{OPS},
and sparse domination, see Barron--Pipher \cite{BP}.

The multi-parameter harmonic analysis has a very rich and renowned history.
For example, we mention the famous covering theorem of Journ\'e \cite{Jo1}, and the deep product BMO, Hardy space and multi-parameter singular integral theory by 
Chang and Fefferman \cite{CF1}, \cite{CF2}, Fefferman \cite{Fe}, Fefferman and Stein \cite{FS} and Journ\'e \cite{Jo2}. Some more
recent references include Ferguson--Lacey \cite{FL}, Pipher--Ward \cite{PW} and Treil \cite{Tr}. The importance of dyadic techniques is already very apparent from these references. Our methods are most clearly tied to the state of the art dyadic--probabilistic developments and representation theorems.

In this paper we consider the most complicated dyadic model operators of modern dyadic multi-parameter representation theorems,
and prove new bounds for them via operator-valued dyadic shifts. The use of operator-valued dyadic shifts in this context
is a new and useful viewpoint even if we would be just considering scalar-valued theory. However, we can even work
with $E$-valued functions, where $E$ is a UMD function lattice (or sometimes even a general UMD space satisfying Pisier's property $(\alpha)$).
The proved bounds for the model operators translate into new results for multi-parameter singular integrals.
For example, we prove the following theorem.
\begin{thm}
Let $E$ be a UMD function lattice, and $p,q \in (1,\infty)$. Let $T$ be a bi-parameter singular integral satisfying $T1$ type assumptions
as in \cite{Ma1}. Then we have
$$
\|T\|_{L^q(\R^n; L^p(\R^m;E)) \to L^q(\R^n; L^p(\R^m;E))} < \infty.
$$
\end{thm}
In fact, we show this in the so called $\calR$-boundedness sense
for all suitable families of bi-parameter singular integrals. This extends the $T1$ type corollary of \cite{Ma1} already in three ways:
we can consider UMD function lattices (instead of $\R$ or $\C$), we get $L^q(L^p)$ boundedness for all $p,q \in (1,\infty)$ instead of just $L^2$ boundedness,
and we get $\calR$-boundedness results for families of bi-parameter singular integrals.

Multiparameter singular integrals have previously been studied in the UMD-valued setting
only in ``paraproduct free'' situations (this means that both the full and partial paraproducts disappear).
For example, Hyt\"onen and Portal \cite{HP} consider convolution-type (and hence paraproduct free) singular integrals. 
Di Plinio and Ou \cite{DO} have also studied $T1$ theorems in paraproduct free situations. For the first time, we do not have such limitations.

In the bi-parameter setting the paper \cite{Ma1} identifies
various types of paraproducts: two different full paraproducts and many partial paraproducts. These are very different in nature; the full paraproducts are
related to the so-called product BMO space, while the partial paraproducts have, in a sense, a paraproduct component only in one of the parameters. 
However, when we increase the amount of parameters, the full paraproducts of the previous generations start to appear in the new partial paraproducts. For example,
in the tri-parameter case we encounter partial paraproducts, which contain full bi-parameter paraproducts in a complicated way. Therefore, it is really the partial paraproducts
which are at the heart of the multi-parameter representation theorems: to deal with them, you need to deal with the full paraproducts.

The $L^2$ theory of partial paraproducts is not terribly complicated -- for the bi-parameter case see \cite{Ma1}. The bi-parameter $L^p$ estimates are established at least in \cite{HPW} using shifted square functions as the tool. Here we offer a new approach via operator-valued shifts. It allows to deal with more complicated and abstract vector-valued operators, gives $L^p(L^q)$ type bounds in a natural way, and is relatively explicit in the way it can handle arbitrary parameters. 
After the abstract theory, the application
to partial paraproducts only requires the verification of the $\calR$-boundedness of some families of full multi-parameter paraproducts. 
These results are of independent interest.

The abstract results are proved using operator-valued dyadic shifts as the main tool (in the applications we have some paraproduct-valued shifts).
The boundedness of operator-valued shifts was shown by H\"anninen--Hyt\"onen \cite{HH} in the one-parameter case.
We start by a small extension of this result by showing the $\calR$-boundedness of families of operator-valued shifts. This is extremely useful as it, for example, allows
us to extend this result to the bi-parameter situation by using shift-valued shifts. However, we also further develop the theory of operator-valued shifts in other subtle ways, which
is crucial for us when we consider some mixed-norm estimates appearing in the applications.
For clarity, three parameters is the highest degree of parameters that we tackle explicitly.

\subsection*{Acknowledgements}
T. Hyt\"onen was supported by the Finnish Centre of Excellence in Analysis and Dynamics Research. 
H. Martikainen is supported by the Academy of Finland through the grants 294840 and 306901, and is a member of the Finnish Centre of Excellence in Analysis and Dynamics Research.
E. Vuorinen is supported by the Academy of Finland through the grant 306901 and by the Finnish Centre of Excellence in Analysis and Dynamics Research.

\section{Definitions and preliminaries}
\subsection{Vinogradov notation}
We denote $A \lesssim B$ if $A \le CB$ for some absolute constant $C$. The constant $C$ can at least depend on the dimensions of the appearing Euclidean spaces, on integration exponents and
on Banach space constants (UMD and Pisier's $(\alpha)$).
We denote $A \sim B$ if $B \lesssim A \lesssim B$.

\subsection{Dyadic notation}
If $Q$ is a cube:
\begin{itemize} 
\item $\ell(Q)$ is the side-length of $Q$;
\item $\text{ch}(Q)$ denotes the dyadic children of $Q$;
\item If $Q$ is in a dyadic grid, then $Q^{(k)}$ denotes the unique dyadic cube $S$ in the same grid so that $Q \subset S$ and $\ell(S) = 2^k\ell(Q)$;
\item If $\calD$ is a dyadic grid, then $\calD_i = \{Q \in \calD\colon\, \ell(Q) = 2^{-i}\}$.
\end{itemize}

In this paper we denote a dyadic grid in $\R^n$ by $\calD^n$. We are at most working in the tri-parameter setting, and then we will have
three dyadic grids $\calD^n$, $\calD^m$, $\calD^k$ in $\R^n$, $\R^m$ and $\R^k$ respectively. Using the above notation $\calD^n_i$ denotes
those $I \in \calD^n$ for which $\ell(I) = 2^{-i}$. The measure of a cube $I$ is simply denoted by $|I|$ no matter in what dimension we are in.

When $I \in \calD^n$ we denote by $h_I$ a cancellative $L^2$ normalised Haar function. This means the following.
Writing $I = I_1 \times \cdots \times I_n$ we can define the Haar function $h_I^{\eta}$, $\eta = (\eta_1, \ldots, \eta_n) \in \{0,1\}^n$, by setting
\begin{displaymath}
h_I^{\eta} = h_{I_1}^{\eta_1} \otimes \cdots \otimes h_{I_n}^{\eta_n}, 
\end{displaymath}
where $h_{I_i}^0 = |I_i|^{-1/2}1_{I_i}$ and $h_{I_i}^1 = |I_i|^{-1/2}(1_{I_{i, l}} - 1_{I_{i, r}})$ for every $i = 1, \ldots, n$. Here $I_{i,l}$ and $I_{i,r}$ are the left and right
halves of the interval $I_i$ respectively. If $\eta \ne 0$ the Haar function is cancellative: $\int h_I^{\eta} = 0$. We usually suppress the presence of $\eta$
and simply write $h_I$ for some $h_I^{\eta}$, $\eta \ne 0$.

For $I \in \calD^n$ and a locally integrable function $f\colon \R^n \to E$, where $E$ is a Banach space, we define the martingale difference
$$
\Delta_I f = \sum_{I' \in \textup{ch}(I)} \big[ \bla f \bra_{I'} -  \bla f \bra_{I} \big] 1_{I'}.
$$
Here $\bla f \bra_I = \frac{1}{|I|} \int_I f$ (where the integral is the usual $E$-valued Bochner integral). 
Then $\Delta_I f = \sum_{\eta \ne 0} \langle f, h_{I}^{\eta}\rangle h_{I}^{\eta}$, or suppressing the $\eta$ summation, $\Delta_I f = \langle f, h_I \rangle h_I$.
Here $\langle f, h_I \rangle = \int f h_I$. In this paper the brackets $\langle \cdot, \cdot \rangle$ try to always refer to some kind of integral pairing, while
$\{ \cdot, \cdot \}_E$ is used for the dual pairing of a Banach space $E$.

A martingale block is defined by
$$
\Delta_K^i f = \mathop{\sum_{I \in \calD^n}}_{I^{(i)} = K} \Delta_I f, \qquad K \in \calD^n.
$$

\subsection{Multi-parameter notation}
We work either in the bi-parameter setting in the product space $\R^{n+m}$ or in the tri-parameter setting in the product space $\R^{n+m+k}$.
In such a context $x$ (or $y$) is always a tuple, for example if $x \in \R^{n+m+k}$, then $x = (x_1, x_2, x_3)$ with $x_1 \in \R^n$, $x_2 \in \R^m$ and $x_3 \in \R^k$.

We often need to take integral pairings with respect to one or two of the variables only.
For example, if $f \colon \R^{n+m+k} \to E$, where $E$ is a Banach space, then $\langle f, h_I \rangle_1 \colon \R^{m+k} \to E$ is defined by
$$
\langle f, h_I \rangle_1(x_2, x_3) = \int_{\R^n} f(y_1, x_2, x_3)h_I(y_1)\ud y_1,
$$
and 
$\langle f, h_I \otimes h_J \rangle_{1,2} \colon \R^{k} \to E$ is defined by
$$
\langle f, h_I \otimes h_J \rangle_{1,2}(x_3) = \int_{\R^m} \int_{\R^n} f(y_1, y_2, x_3)h_I(y_1)h_J(y_2) \ud y_1 \ud y_2.
$$
Moreover, an identification of the following kind is used all the time:  a function $f \colon \R^{n+m} \to E$ satisfying
$$
\Big(\int_{\R^n} \Big( \int_{\R^m} |f(x_1, x_2)|_E^p \ud x_2 \Big)^{q/p} \ud x_1\Big)^{1/q} < \infty
$$
is identified with the function $\phi_f \in L^q(\R^n; L^p(\R^m;E))$, $\phi_f(x_1) = f(x_1, \cdot)$.

We next define bi-parameter martingale differences. Let $f \colon \R^n \times \R^m \to E$ be locally integrable.
Let $I \in \calD^n$ and $J \in \calD^m$. We define the martingale difference
$$
\Delta_I^1 f \colon \R^{n+m} \to E, \Delta_I^1 f(x) := \Delta_I (f(\cdot, x_2))(x_1).
$$
(The reader should not confuse this with the martingale block notation of a one-parameter function from above).
Define $\Delta_J^2f$ analogously. Then we set
$$
\Delta_{I \times J} f \colon \R^{n+m} \to E, \Delta_{I \times J} f(x) = \Delta_I^1(\Delta_J^2 f)(x) = \Delta_J^2 ( \Delta_I^1 f)(x).
$$
Notice that $\Delta^1_I f = h_I \otimes \langle f , h_I \rangle_1$, $\Delta^2_J f = \langle f, h_J \rangle_2 \otimes h_J$ and
$ \Delta_{I \times J} f = \langle f, h_I \otimes h_J\rangle h_I \otimes h_J$ (suppressing the finite $\eta$ summations).

Martingale blocks are defined in the natural way
$$
\Delta_{K \times V}^{i, j} f  =  \sum_{I\colon I^{(i)} = K} \sum_{J\colon J^{(j)} = V} \Delta_{I \times J} f = \Delta_{K,i}^1( \Delta_{V,j}^2 f) = \Delta_{V,j}^2 ( \Delta_{K,i}^1 f).
$$
\subsection{BMO spaces}\label{ss:bmo}

We say that $b \in L^1_{\loc}(\R^n)$ belongs to the dyadic BMO space $\BMO_{\calD^n}(\R^n) = \BMO_{\calD^n}$ if
$$
\|b\|_{\BMO_{\calD^n}} := \sup_{I \in \calD^n} \frac{1}{|I|} \int_I |b - \langle b \rangle_I| < \infty.
$$
The ordinary space $\BMO(\R^n)$ is defined by taking the supremum over all cubes.

\subsubsection*{Bi-parameter product BMO}
Here we define the (dyadic) bi-parameter
product BMO space $\BMO_{\textup{prod}}^{\calD^n, \calD^m}(\R^n \times \R^m) = \BMO_{\textup{prod}}^{\calD^n, \calD^m}$.
For a sequence $\lambda = (\lambda_{I,J})$ we set
\begin{equation*}
\|\lambda\|_{\BMO_{\textup{prod}}^{\calD^n, \calD^m}} := 
\sup_{\Omega} \Big( \frac{1}{|\Omega|} \mathop{\sum_{I \in \calD^n, J \in \calD^m}}_{I \times J \subset \Omega} |\lambda_{I,J}|^2 \Big)^{1/2},
\end{equation*}
where the supremum is taken over those sets $\Omega \subset \R^{n+m}$ such that $|\Omega| < \infty$ and such that for every $x \in \Omega$ there exists
$I \in \calD^n, J \in \calD^m$ so that $x \in I \times J \subset \Omega$. 

We say that $b \in L^1_{\loc}(\R^{n+m})$ belongs to the space $\BMO_{\textup{prod}}^{\calD^n, \calD^m}$ if
$$
\| b \|_{\BMO_{\textup{prod}}^{\calD^n, \calD^m}} := \| (\langle b, h_I \otimes h_J\rangle)_{I,J} \|_{\BMO_{\textup{prod}}^{\calD^n, \calD^m}} < \infty.
$$
The (non-dyadic) product BMO space $\BMO_{\textup{prod}}(\R^{n+m})$ can be defined via the norm defined by the supremum of
the above dyadic norms.

For two sequences $\lambda = (\lambda_{I,J})$, $A = (A_{I, J})$ we have the key estimate
$$
\sum_{I \in \calD^n, J \in \calD^m} |\lambda_{I,J}| |A_{IJ}| \lesssim \|\lambda\|_{\BMO_{\textup{prod}}^{\calD^n, \calD^m}} \|S_{\calD^n, \calD^m}(A)\|_{L^1(\R^{n+m})},
$$
where
$$
S_{\calD^n, \calD^m}(A) := \Big( \sum_{I \in \calD^n, J \in \calD^m} |A_{IJ}|^2 \frac{1_{I \times J}}{|I \times J|} \Big)^{1/2}.
$$
For a simple proof see e.g. Proposition 4.1 of \cite{MO}. This inequality is the key property of the product BMO for us.
Of course, an analogous estimate holds in the one-parameter situation.


\subsection{Paraproducts}
Let $E$ be a Banach space.
A function $b \in L^1_{\loc}(\R^n)$ defines the dyadic paraproduct by the formula
$$
\pi_{\calD^n, b}f = \sum_{I \in \calD^n}  \langle f \rangle_I \Delta_I b, \qquad f \in L^1_{\loc}(\R^n; E).
$$
\subsubsection*{Bi-parameter full paraproducts}
While our main object of study in this paper are the so called partial paraproducts, we will also need to consider the so called full bi-parameter
paraproducts. For example, they appear in some of the partial tri-parameter paraproducts.

There are two types of full bi-parameter paraproducts: the standard ones and the mixed ones:
for $b \in L^1_{\loc}(\R^{n+m})$ we set
\begin{align*}
\Pi_{\calD^n, \calD^m, b} f &= \mathop{\sum_{I \in \calD^n}}_{J \in \calD^m} \langle b, h_I \otimes h_J \rangle \langle f \rangle_{I \times J} h_I \otimes h_J; \\
\Pi_{\calD^n, \calD^m,b}^{\textup{mixed}} f &= \mathop{\sum_{I \in \calD^n}}_{J \in \calD^m} \langle b, h_I \otimes h_J \rangle \Big\langle f, h_I \otimes \frac{1_J}{|J|} \Big\rangle
\frac{1_I}{|I|} \otimes h_J, \qquad f \in L^1_{\loc}(\R^{n+m};E).
\end{align*}



\subsection{UMD and Pisier's property $(\alpha)$}
A Banach space $E$ is said to be a UMD space if
$$
\Big\| \sum_{i=1}^N \epsilon_i d_i \Big\|_{L^p(\Omega; E)} \lesssim \Big\| \sum_{i=1}^N d_i \Big\|_{L^p(\Omega; E)}
$$
for all $E$-valued $L^p$-martingale difference sequences $(d_i)_{i=1}^N$ (defined on some probability space $\Omega$),
and for all signs $\epsilon_i \in \{-1 +1\}$. The UMD property is independent of the choice of the exponent $p \in (1,\infty)$.
If $E$ is UMD then so is $E^*$ and $L^p(\R^n;E)$.
 
A Banach space $E$ has Pisier's property $(\alpha)$ if for all $N$, all $\alpha_{i,j}$ in the complex unit disc and all
$e_{i,j} \in E$, $1 \le i, j \le N$, there holds
$$
\Big(\E \E' \Big| \sum_{1 \le i, j \le N} \epsilon_i \epsilon_j' \alpha_{i,j} e_{i,j} \Big|_E^2\Big)^{1/2} \lesssim \Big(\E \E' \Big| \sum_{1 \le i, j \le N} \epsilon_i \epsilon_j' e_{i,j} \Big|_E^2 \Big)^{1/2}.
$$ 
Here $(\epsilon_i)$ and $(\epsilon_j')$ are sequences of independent random signs.
If $E$ has Pisier's property $(\alpha)$ then so does $L^p(\R^n;E)$.

The Kahane--Khintchine inequality says that
$$
\Big( \E \Big| \sum_{i=1}^N \epsilon_i e_i \Big|_E^q \Big)^{1/q} \sim_q \Big( \E \Big| \sum_{i=1}^N \epsilon_i e_i \Big|_E^2 \Big)^{1/2}
$$
for all $1 \le q < \infty$ and Banach spaces $E$. By a few applications of the  Kahane--Khintchine inequality we see that we can use whatever exponent in the definition of property $(\alpha)$.
\subsection{$\calR$-boundedness}
If $E$ and $F$ are Banach spaces, we denote the space of bounded linear operators from $E$ to $F$ by $\calL(E, F)$. If $E = F$ we simply write $\calL(E)$.
A family of operators $\mathcal{T} \subset \calL(E, F)$ is said to be $\calR$-bounded if
for all $N$, $T_1, \ldots, T_N \in \mathcal{T}$ and $e_1, \ldots, e_N \in E$ we have
\begin{equation*}
\Big(\E \Big| \sum_{i=1}^N \epsilon_i T_i e_i \Big|^2_F\Big)^{1/2} \le C \Big( \E \Big| \sum_{i=1}^N \epsilon_i e_i \Big|_E^2 \Big)^{1/2}.
\end{equation*}
The best constant $C$ is denoted by $\calR(\mathcal{T})$. The Kahane--Khintchine inequality shows that one can replace in the definition the exponent
$2$ with any $q \in [1, \infty)$.


\subsection{Random sums and duality}
For the definition of type and cotype of a Banach space the reader can e.g. consult the section 7 of the book \cite{HNVW2}. 
However, for the following lemma a reader not familiar with the notion of type needs only to know that
UMD spaces have non-trivial type (as all spaces of interest in this paper will be UMD).
See Section 7.4.f in \cite{HNVW2} for a proof.
\begin{lem}\label{lem:randomduality}
Let $E$ be a Banach space with non-trivial type and let $F \subset E^*$ be a closed subspace of $E^*$ which is norming for $E$. Let $p \in (1,\infty)$. Then
for all finite sequences $e_1, \ldots, e_N \in E$ we have
$$
\Big( \E \Big| \sum_{i=1}^N \epsilon_i e_i \Big|_E^p \Big)^{1/p} \lesssim \sup \Big\{ \Big| \sum_{i=1}^N \{e_i, e_i^*\}_E \Big|\Big\},
$$
where the supremum is taken over all choices $(e_i^*)_{i=1}^N$ in $F$ such that
$$
\Big( \E \Big| \sum_{i=1}^N \epsilon_i e_i^* \Big|_{E^*}^{p'} \Big)^{1/{p'}} \le 1. 
$$
The converse inequality trivially holds with a constant $1$.
\end{lem}

\subsection{Operator-valued shifts}
Here we give the definition of ordinary (i.e. one-parameter) operator-valued shifts as given by H\"anninen--Hyt\"onen \cite{HH}.
Let $E$ be a UMD space and $i_1, i_2 \ge 0$ be two indices. An operator-valued shift $S^{i_1,i_2}_{\calD^n}$ in $\R^n$ (defined using
a fixed dyadic grid $\calD^n$) is an operator of the form
$$
S^{i_1,i_2}_{\calD^n} f = \sum_{K \in \calD^n} \Delta_K^{i_2} A_{K} \Delta_K^{i_1} f, \qquad f \in L^1_{\loc}(\R^n;E),
$$
where $A_{K} f \colon \R^n \to E$ is an averaging operator with an operator-valued kernel
$a_K \colon \R^n \times \R^n \to \calL(E)$, that is,
$$
A_{K} f(x) = \frac{1_K(x)}{|K|} \int_K a_K(x, y) f(y)\ud y. 
$$
Provided that the family of kernels is $\calR$-bounded, i.e., 
$$
\calR(\{a_K(x, y) \in \calL(E)\colon K \in \calD^n, x, y \in K\}) \le C_a,
$$
H\"anninen--Hyt\"onen \cite{HH} proved that for all $1 < q < \infty$ we have
$$
\|S^{i_1,i_2}_{\calD^n} f \|_{L^q(\R^n; E)} \lesssim (\max(i_1, i_2)+1)C_a \|f\|_{L^q(\R^n;E)}.
$$
The implicit constant depends on the UMD-constant of $E$ and on $q$. The result actually holds with $\min(i_1, i_2)$ in place of $\max(i_1,i_2)$. In this paper we
need to prove the $\calR$-boundedness of shifts (under the assumption that $E$ also has Pisier's property $(\alpha)$), and we
do this with $\min$ in place of $\max$ (see Lemma \ref{lem:ShiftRBound}).

\subsubsection*{Operator-valued bi-parameter shifts}\label{ss:opbipardef}
Let $E$ be a UMD space satisfying the property $(\alpha)$ of Pisier.
An operator-valued bi-parameter dyadic shift in $\R^{n+m}$ with parameters $i_1$, $i_2$, $j_1$ and $j_2$ is an operator of the form
\begin{equation*}
S^{i_1,i_2,j_1,j_2}_{\calD^n, \calD^m} f 
= \mathop{\sum_{K \in \calD^n}}_{V \in \calD^m} \Delta_{K \times V}^{i_2, j_2} A_{K,V} \Delta_{K \times V}^{i_1, j_1} f, \quad f \in L^1_{\text{loc}}(\R^{n+m}; E).
\end{equation*}
Here each $A_{K,V}$ is an integral operator related to a kernel 
$$
a_{K,V} \colon \R^{n+m} \times \R^{n+m} \to \calL(E)
$$
by
$$
A_{K,V}f(x) =\frac{1_{K\times V}}{|K| |V|} \iint_{K \times V}a_{K,V}(x,y) f(y) \ud y.
$$
The family of kernels is assumed to be $\calR$-bounded in the sense that
\begin{equation*}
\calR\big(\{ a_{K,V}(x,y) \colon K \in \calD^n, V \in \calD^m, x,y \in K\times V \}\big) \le C_a.
\end{equation*}
We will (among other things) prove in Section \ref{sec:biopshift} that for all $p, q \in (1,\infty)$ we have
$$
\| S^{i_1,i_2,j_1,j_2}_{\calD^n, \calD^m} f \|_{L^q(\R^n; L^p(\R^m;E))}
\lesssim (\min(i_1, i_2)+1)(\min(j_1, j_2)+1)C_a \|f\|_{L^q(\R^n; L^p(\R^m;E))}.
$$

\subsection{Function lattices}\label{ss:fl}
An easy to read reference for this section is \cite{Lo} (see also \cite{LT} and \cite{ALV}).
A normed space $E$ is a Banach function space (or a function lattice) if the following four conditions hold.
Let $(\Omega, \mathcal{A}, \mu)$ be a $\sigma$-finite measure space.
\begin{enumerate}
\item Every $f \in E$ is a measurable function $f \colon \Omega \to \R$ (an equivalence class).
\item If $f \colon \Omega \to \R$ is measurable, $g \in E$ and $|f(\omega)| \le |g(\omega)|$ for $\mu$-a.e. $\omega \in \Omega$, then
$f \in E$ and $|f|_E \le |g|_E$.
\item There is an element $f \in E$ so that $f > 0$ (i.e. $f(\omega) > 0$ for $\mu$-a.e. $\omega \in \Omega$).
\item If $f_i ,f$ are non-negative, $f_i \in E$, $f_i \le f_{i+1}$, $f_i(\omega) \to f(\omega)$ for $\mu$-a.e. $\omega \in \Omega$ and $\sup_i |f_i|_E < \infty$,
we have $f \in E$ and $|f_i|_E \to |f|_E = \sup_i |f_i|_E$.
\end{enumerate}
\begin{rem}\label{rem:vague}
The definition of a function lattice seems to vary a little bit in the literature, and sometimes it is left quite vague what is the exact definition used.
Here we use the definition from \cite{ALV} and \cite{Lo}.
\end{rem}
Such a space is automatically a Banach space, and in fact a Banach lattice (for the definition and basic theory of Banach lattices see e.g. \cite{LT}).
If $g \colon \Omega \to \R$ is a measurable function such that $fg \in L^1(\mu)$ for all $f \in E$, we define
$$
e_g^* \colon E \to \R, e_g^*(f) = \int_{\Omega} f(\omega)g(\omega) \ud \mu(\omega).
$$
In this case $e_g^* \in E^*$. We define $E' \subset E^*$ to consist of those element of $E^*$ that have the form $e_g^*$ for some $g$ like above (and
we freely identify $e_g^*$ with $g$). The space $E'$ -- the K\"othe dual of $E$ -- is a Banach function space, and a norming subspace of $E^*$.

There is a condition called order continuity (the precise definition does not interest us here) of $E$ which is equivalent with the fact that $E' = E^*$. So the dual of an order continuous
Banach function space is a Banach function space (as $E'$ is always a Banach function space). We will be working with UMD Banach function spaces --
which we also call UMD function lattices. Such spaces are always reflexive,
and reflexive Banach lattices are order continuous. So in cases of interest to us $E' = E^*$. In this case $E$ automatically also satisfies the property $(\alpha)$ of Pisier.
This is because a Banach lattice satisfies Pisier's property $(\alpha)$ if and only if it has finite cotype (see e.g. Theorem 7.5.20 in the book \cite{HNVW2}). 
A UMD space certainly has finite cotype (see e.g. \cite{HNVW2}).

This generality covers most of the naturally arising examples of UMD spaces satisfying the property $(\alpha)$ of Pisier.
The only place where we really need that E is a function space (and not just a general UMD space satisfying Pisier's property $(\alpha)$) is Section \ref{sec:RbounbforPar}, where we prove $\calR$-boundedness results for various families of paraproducts.
That is to say, if one can generalise the results of Section \ref{sec:RbounbforPar}, then this restriction can be lifted in other key results.

A key reason why we use function lattices instead of Banach lattices is because we want to use maximal function estimates by Bourgain \cite{Bo} and
Rubio de Francia \cite{Ru}. Function lattices are certainly more convenient to use in other aspects too, but it could be the case that
most other estimates could be performed in Banach lattices.

\subsection{Estimates for martingales}\label{ss:estimates}
We collect in this subsection a plethora of various estimates, many of them of standard nature.
The main aim is Corollary \ref{cor:biparmarSEQ}.
The results are stated in the bi-parameter case.
For clarity we give some proofs.
\begin{lem}\label{lem:biparmar1}
Let $E$ be a UMD space. Then for all $p, q \in (1, \infty)$ and fixed signs $\epsilon_I$, $\epsilon_J$ we have
$$
\|f\|_{L^q(\R^n; L^p(\R^m; E))} \sim \Big\| \sum_{I \in \calD^n} \epsilon_I \Delta_I^1 f \Big\|_{L^q(\R^n; L^p(\R^m; E))}
\sim \Big\| \sum_{J \in \calD^m} \epsilon_J \Delta_J^2 f \Big\|_{L^q(\R^n; L^p(\R^m; E))}.
$$
\end{lem}
\begin{proof}
The first one is seen by using the fact that $F := L^p(\R^m; E)$ is a UMD space and expanding
the function $f \colon \R^n \to F$. The second one is seen by expanding $f(x_1, \cdot) \colon \R^m \to E$
for each fixed $x_1 \in \R^n$.
\end{proof}
\begin{lem}\label{lem:biparmar}
Let $E$ be a UMD space satisfying Pisier's property $(\alpha)$. Then
for all $p, q \in (1, \infty)$ and fixed signs $\epsilon_{I,J}$ we have
$$
\|f\|_{L^q(\R^n; L^p(\R^m; E))} \sim  \Big\|\mathop{\sum_{I \in \calD^n}}_{J \in \calD^m} \epsilon_{I, J} \Delta_{I \times J} f\Big\|_{L^q(\R^n; L^p(\R^m;E))}.
$$
\end{lem}
\begin{proof}
Using Lemma \ref{lem:biparmar1} twice and taking expectations we get
\begin{equation}\label{eq:eq1}
\|f\|_{L^q(\R^n; L^p(\R^m;E))} 
\sim \E\E' \Big\| \sum_{I \in \calD^n} \sum_{J \in \calD^m} \epsilon_I \epsilon_J' \Delta_{I \times J} f \Big\|_{L^q(\R^n; L^p(\R^m; E))}.
\end{equation}
Using the fact that $L^q(\R^n; L^p(\R^m;E))$ satisfies Pisier's property $(\alpha)$ we get
\begin{align*}
\E\E' \Big\| \sum_{I \in \calD^n}& \sum_{J \in \calD^m} \epsilon_I \epsilon_J' \Delta_{I \times J} f \Big\|_{L^q(\R^n; L^p(\R^m; E))} \\
&\sim \E\E' \Big\| \sum_{I \in \calD^n} \sum_{J \in \calD^m} \epsilon_I \epsilon_J' \epsilon_{I, J} \Delta_{I \times J} f \Big\|_{L^q(\R^n; L^p(\R^m; E))}.
\end{align*}
Applying the identity \eqref{eq:eq1} to the function $F = \sum_{I \in \calD^n} \sum_{J \in \calD^m} \epsilon_{I, J} \Delta_{I \times J} f$ we get the claim.
\end{proof}
\begin{lem}
Let $E$ be a UMD space satisfying Pisier's property $(\alpha)$. Then
for all $p, q \in (1, \infty)$ we have
\begin{align*}
\E \Big\| \sum_j \epsilon_j f_j \Big\|_{L^q(\R^n; L^p(\R^m;E))} &\sim 
\E \Big\|\sum_j \mathop{\sum_{I \in \calD^n}}_{J \in \calD^m} \epsilon_{j,I, J} \Delta_{I \times J} f_j \Big\|_{L^q(\R^n; L^p(\R^m;E))} \\
&\sim \E \Big\| \sum_j \sum_{I \in \calD^n} \epsilon_{j, I} \Delta_I^1 f_j \Big\|_{L^q(\R^n; L^p(\R^m;E))} \\
&\sim \E \Big\| \sum_j \sum_{J \in \calD^m} \epsilon_{j, J} \Delta_J^2 f_j \Big\|_{L^q(\R^n; L^p(\R^m;E))}.
\end{align*}
\end{lem}
\begin{proof}
We only prove the first estimate -- the others are proved in the same way (just use Lemma \ref{lem:biparmar1} instead of Lemma \ref{lem:biparmar}).
Using Lemma \ref{lem:biparmar} to the function $\sum_j \epsilon_j f_j$ and taking expectations we get that
$$
\E \Big\| \sum_j \epsilon_j f_j \Big\|_{L^q(\R^n; L^p(\R^m;E))}
\sim \E \E' \Big\|\sum_j \mathop{\sum_{I \in \calD^n}}_{J \in \calD^m} \epsilon_j \epsilon_{I, J}' \Delta_{I \times J} f_j \Big\|_{L^q(\R^n; L^p(\R^m;E))}.
$$
Using the fact that $L^q(\R^n; L^p(\R^m;E))$ satisfies Pisier's property $(\alpha)$ we get
\begin{align*}
\E \E' \Big\|\sum_j \mathop{\sum_{I \in \calD^n}}_{J \in \calD^m} \epsilon_j \epsilon_{I, J}' &\Delta_{I \times J} f_j \Big\|_{L^q(\R^n; L^p(\R^m;E))} \\
&\sim \E \E' \E'' \Big\|\sum_j \mathop{\sum_{I \in \calD^n}}_{J \in \calD^m} \epsilon_j \epsilon_{I, J}' \epsilon_{j, I, J}'' \Delta_{I \times J} f_j \Big\|_{L^q(\R^n; L^p(\R^m;E))} \\
&= \E'' \Big\|\sum_j \mathop{\sum_{I \in \calD^n}}_{J \in \calD^m} \epsilon_{j, I, J}'' \Delta_{I \times J} f_j \Big\|_{L^q(\R^n; L^p(\R^m;E))}.
\end{align*}
\end{proof}
In the case that $E$ is a UMD function lattice (notice that in this case $|e| \in E$ and $|e|^{\alpha} \in E$ are defined in the natural pointwise way for every $e \in E$)
we prefer square function bounds. The previous estimates are translated to them using the following lemma. The important thing for us
is that it holds with all UMD function lattices.
\begin{lem}\label{lem:KKlemma}
Let $E$ be a Banach function space with finite cotype.
For all $q \in (1,\infty)$ we have
$$
\Big\| \Big(\sum_j |f_j|^2 \Big)^{1/2} \Big\|_{L^q(\R^n; E)} \sim \E \Big\| \sum_j \epsilon_j f_j \Big\|_{L^q(\R^n; E)}.
$$
\end{lem}
\begin{proof}
Applying Kahane--Khintchine inequality multiple times we see that
$$
\E \Big\| \sum_j \epsilon_j f_j \Big\|_{L^q(\R^n; E)} \sim \Big\| \Big( \E \Big| \sum_j \epsilon_j f_j \Big|_E^2\Big)^{1/2} \Big\|_{L^q(\R^n; E)}.
$$
So things boil down to the equivalence
\begin{equation}\label{eq:bd}
\Big( \E \Big| \sum_j \epsilon_j e_j \Big|_E^2\Big)^{1/2} \sim \Big|\Big(\sum_j |e_j|^2 \Big)^{1/2}\Big|_E
\end{equation}
for all $e_j \in E$. But \eqref{eq:bd} holds in all Banach lattices with finite cotype -- this is a theorem of Khintchine--Maurey. For a proof see
Theorem 7.2.13 in \cite{HNVW2}.
\end{proof}
\begin{cor}\label{cor:biparmarSEQ}
Let $E$ be a UMD function lattice.
For all $p, q \in (1, \infty)$ we have
\begin{align*}
\Big\| \Big(\sum_j |f_j|^2 \Big)^{1/2} \Big\|_{L^q(\R^n; L^p(\R^m;E))} 
&\sim \Big\| \Big( \sum_j \mathop{\sum_{I \in \calD^n}}_{J \in \calD^m} |\Delta_{I \times J} f_j|^2 \Big)^{1/2} \Big\|_{L^q(\R^n; L^p(\R^m;E))} \\
&\sim  \Big\| \Big( \sum_j \sum_{I \in \calD^n} |\Delta_I^1 f_j|^2 \Big)^{1/2} \Big\|_{L^q(\R^n; L^p(\R^m;E))} \\
&\sim \Big\| \Big( \sum_j \sum_{J \in \calD^m} |\Delta_J^2 f_j|^2 \Big)^{1/2} \Big\|_{L^q(\R^n; L^p(\R^m;E))}.
\end{align*}
\end{cor}

\subsection{Estimates for maximal functions}\label{ss:maxestimates}
We introduce the lattice maximal function theory of Bourgain \cite{Bo} and Rubio de Francia \cite{Ru}.
Let $E$ be a Banach function space. For each (simple) locally integrable function $f \colon \R^n \to E$ we define
$M_{\calD^n, E}f \colon \R^n \to E$ by setting
$$
M_{\calD^n, E}f(x, \omega) = \mathop{\sup_{I \in \calD^n}}_{x \in I} \frac{1}{|I|} \int_I |f(y, \omega)|\,dy.
$$

The following definitions are in line with our usual notational conventions.
If $f \colon \R^{n+m} \to E$
we define $M_{\calD^n, E}^1f \colon \R^{n+m} \to E$ by setting $M^1_{\calD^n, E} f(x_1, x_2)
=  M_{\calD^n, E}(f(\cdot, x_2))(x_1)$. The operator $M^2_{\calD^m, E}$ is defined similarly.

The bi-parameter strong maximal function $\calM_{\calD^n, \calD^m, E}$ is defined for $f \colon \R^{n+m} \to E$ by setting
$$
\calM_{\calD^n, \calD^m, E}f(x_1,x_2, \omega) = \mathop{\sup_{I \in \calD^n}}_{J \in \calD^m} \frac{1_I(x_1)1_J(x_2)}{|I||J|} \iint_{I \times J} |f(y_1, y_2, \omega)|\ud y_1 \ud y_2.
$$
If $E$ is the scalar field we simply write $M_{\calD^n}$ etc.

The following proposition is due to Bourgain \cite{Bo} and Rubio de Francia \cite{Ru} in the case of the torus. A weighted version of this in $\R^n$ is proved in \cite{GMT}.
See also \cite{ALV} and \cite{Lo}. Again, the point made in Remark \ref{rem:vague} applies here.
\begin{prop}\label{prop:Rubio}
Let $E$ be a UMD function lattice. Then for all $q \in (1,\infty)$ we have
$$
\|M_{\calD^n, E}f\|_{L^q(\R^n; E)} \lesssim \|f\|_{L^q(\R^n; E)}, \qquad f \in L^q(\R^n;E).
$$
\end{prop}
\begin{rem}
If $E$ is a UMD function lattice, then so is $E^* = E'$.
Therefore, we also have for each $q \in (1,\infty)$ that
$$
\|M_{\calD^n, E^*}g\|_{L^q(\R^n; E^*)} \lesssim \|g\|_{L^q(\R^n; E^*)}, \qquad g \in L^q(\R^n;E^*).
$$
\end{rem}

Proposition \ref{prop:Rubio} extends for instance to the case where the function lattice $E$ is replaced with the function lattice
$L^p(\R^m;l^r(E))$, where $p,r \in (1, \infty)$.
Let $q \in (1, \infty)$. A function $f \in L^q(\R^n; L^p(\R^m;\ell^r(E)))$ can be viewed as
$f=\{f_j\}_{j \in \Z}$, where each $f_j$ is an $E$-valued function on $\R^{n+m}$, and 
$$
\|f\|_{L^q(\R^n; L^p(\R^m;\ell^r(E)))}
= \Big \| \Big( \sum_j |f_j|^r\Big)^{1/r} \Big \|_{L^q(\R^n; L^p(\R^m;E))}.
$$
The maximal function $M_{\calD^n, L^p(\R^m;\ell^r(E))}$ acting on $f$ gives
$$
M_{\calD^n, L^p(\R^m;\ell^r(E))}f=\{M^1_{\calD^n,E}f_j\}_j.
$$
Thus, the extension of Proposition \ref{prop:Rubio} gives the following corollary:

\begin{cor}\label{cor:FS_M1}
Let $E$ be a UMD function lattice.
For all $p, q, r \in (1,\infty)$ we have
\begin{equation*}
\Big \| \Big( \sum_j |M_{\calD^n,E}^1 f_j|^r \Big)^{1/r} \Big\|_{L^q(\R^n; L^p(\R^m;E))}
\lesssim \Big \| \Big( \sum_j | f_j|^r \Big)^{1/r} \Big\|_{L^q(\R^n; L^p(\R^m;E))}.
\end{equation*}
\end{cor}

Corollary \ref{cor:FS_M1} in turn directly leads to Corollary \ref{cor:FS_Strong}:

\begin{cor}\label{cor:FS_Strong}
Let $E$ be a UMD function lattice.
For all $p, q, r \in (1,\infty)$ we have
$$
\Big\| \Big( \sum_j |\mathcal{M}_{\calD^n, \calD^m,E} f_j |^r \Big)^{1/r} \Big\|_{L^q(\R^{n}; L^p(\R^m;E))} \lesssim \Big\| \Big( \sum_{j \in \mathcal{J}} | f_j |^r \Big)^{1/r} \Big\|_{L^q(\R^{n}; L^p(\R^m;E))}.
$$
\end{cor}
\begin{proof}
We first apply the inequality $\mathcal{M}_{\calD^n, \calD^m,E} f_j \le M_{\calD^m,E}^2 M_{\calD^n,E}^1 f_j$. In the inner integral over $\R^m$ we use 
Proposition \ref{prop:Rubio} with $E$ replaced by $\ell^r(E)$  to obtain
$$
\Big\| \Big( \sum_j |\mathcal{M}_{\calD^n, \calD^m,E} f_j |^r \Big)^{1/r} \Big\|_{L^q(\R^{n}; L^p(\R^m;E))} \lesssim 
\Big\| \Big( \sum_j |M_{\calD^n,E}^1 f_j  |^r \Big)^{1/r} \Big\|_{L^q(\R^{n}; L^p(\R^m;E))}.
$$
Corollary \ref{cor:FS_M1} then readily gives the result.
\end{proof}

\section{Various estimates for operator-valued shifts}\label{sec:biopshift}
\subsection{$\calR$-boundedness results for operator-valued shifts}
For various reasons we need the following lemma on the $\calR$-boundedness of one-parameter dyadic shifts. This actually also reproves the boundedness
of operator-valued dyadic shifts from H\"anninen--Hyt\"onen \cite{HH} with a bit better dependence on the pair of indices $(j_1,j_2)$. The proof is similar, though.
However, Lemma \ref{lem:ShiftWrongOrder} offers a more interesting variant, which will also be of key importance to us.

Let us first introduce some definitions related to the so called decoupling estimate.
Let $V \in \calD^m$. 
Denote by $Y_{V}$ the measure space $(V, \text{Leb}_m(V), \nu_{V})$. Here $\text{Leb}_m(V)$ is the collection of 
Lebesgue measurable subset of $V$ and $ \nu_V=\ud x \lfloor V/|V|$, where $\ud x \lfloor V$ is the $m$-dimensional Lebesgue measure restricted to $V$.
With these we define the product probability space
$$
(Y^m, \scrA_m, \nu_m):= \prod_{V \in \calD^m} Y_{V}.
$$
If $y \in Y^m$ and $V \in \calD^m$, we denote by $y_V$ the coordinate related to $Y_V$. 

Suppose $i \in \{0,1,\dots\}$ and $j \in \{0, \dots, i\}$. 
Let $\calD^m_{i,j}$ be the collection
\begin{equation}\label{eq:SubLattice}
\calD^m_{i,j} := \{V \in \calD^m \colon \ell(V)=2^{k(i+1)+j} \text{ for some }k \in \Z\}.
\end{equation}
We have for all $p \in (1, \infty)$ and $f \in L^p(\R^m;E)$, where $E$ is a UMD space, that
\begin{equation}\label{eq:decoupling}
\int_{\R^m} \Big| \sum_{V \in \calD^m_{i,j}} \Delta^i_V f \Big|_E^p \ud x
\sim_p \E \int_{Y^m} \int_{\R^m} 
\Big| \sum_{V \in \calD^m_{i,j}}1_V(x) \Delta^i_V f(y_V) \Big|_E^p \ud x \ud \nu_m (y),
\end{equation}
where the implicit constant is independent of $i$ and $j$. This is a special case of Theorem 3.1 in \cite{HH}.
See also \cite{HH} for the details and the history of this estimate.
The reason why the collections $\calD^m_{i,j}$ are used is to guarantee
that $\Delta^i_V f$ is constant on $V'$ if $V, V' \in \calD^m_{i,j}$ and $V' \subsetneq V$. This is a technical detail needed to apply the underlying
abstract decoupling principle behind \eqref{eq:decoupling}.

\begin{lem}\label{lem:ShiftRBound}
Let $E$ be a UMD space with Pisier's property $(\alpha)$. Fix two parameters $j_1, j_2 \ge 0$.
Suppose $\{S^{j_1,j_2}_{\calD^m,k}\}_{k \in \calK}$ is a family of operator-valued dyadic shifts. For every $k \in \calK$ let $\{a_{V,k}\}_{V \in \calD^m}$
be the family of kernels related to the shift $S^{j_1, j_2}_{\calD^m,k}$.

Assume that there exists a constant $C_a$ so that
\begin{equation}\label{eq:ShiftRAs}
\calR(\{a_{V,k}(x, y) \in \calL(E )\colon k \in \calK, V \in \calD^m, x, y \in V\}) \le C_a.
\end{equation}
Then, for every $q \in (1, \infty)$, 
$$
\calR\big(\{ S^{j_1,j_2}_{\calD^m,k} \in \calL(L^q(\R^m ; E)) \colon k \in \calK\}) \lesssim (\min(j_1,j_2)+1)C_a.
$$
\end{lem}

\begin{proof}
Fix some $q \in (1, \infty)$. Let $\calI \subset \calK$ be a finite subset, and suppose that for every $k \in \calI$ we have a function $f_k \in L^q(\R^m ;E)$. 
The UMD property of $E$ implies that 
\begin{equation}\label{eq:ShiftRStep1}
\begin{split}
\E\Big \| \sum_{k \in \calI} \epsilon_k S^{j_1,j_2}_{\calD^m,k} f_k \Big \|_{L^q(\R^m;E)}
&=\E\Big \| \sum_{V \in \calD^m} \Delta^{j_2}_V  \sum_{k \in \calI} \epsilon_k A_{V,k} \Delta_V^{j_1} f_k \Big \|_{L^q(\R^m;E)} \\
&\sim \E \E'\Big \| \sum_{V \in \calD^m}\epsilon_V' \Delta^{j_2}_V  \sum_{k \in \calI} \epsilon_k A_{V,k} \Delta_V^{j_1} f_k \Big \|_{L^q(\R^m;E)} \\
& \lesssim \E \E'\Big \| \sum_{V \in \calD^m}\epsilon_V'   \sum_{k \in \calI} \epsilon_k A_{V,k} \Delta_V^{j_1} f_k \Big \|_{L^q(\R^m;E)} \\
&\sim \Big(\E \E' \Big \| \sum_{k \in \calI}\sum_{V \in \calD^m} \epsilon_k \epsilon_V' A_{V,k} \Delta_V^{j_1} f_k \Big \|_{L^q(\R^m;E)}^q \Big)^{1/q},
\end{split}
\end{equation}
where we applied Stein's inequality in the second to last step, and the Kahane-Khintchine inequality in the last. The UMD-valued version of Stein's inequality
is by Bourgain, for a proof see e.g. Theorem 4.2.23 in the book \cite{HNVW1}.

Let $V \in \calD^m$ and $k \in \calI$. Applying the probability space $Y^m$, which was introduced in the beginning of this section, we can write 
\begin{equation*}
\begin{split}
A_{V,k} \Delta_V^{j_1} f_k (x) 
&= \frac{1_V(x)}{|V|} \int_V a_{V,k}(x, y) \Delta_V^{j_1} f_k(y) \ud y \\
&= \int_{Y_V} 1_V(x) a_{V,k}(x, y_V) \Delta_V^{j_1} f_k(y_V) \ud \nu_V (y_V) \\
&=\int_{Y^m} 1_V(x) a_{V,k}(x, y_V) \Delta_V^{j_1} f_k(y_V) \ud \nu_m (y).
\end{split}
\end{equation*}
Applying this in the right hand side of \eqref{eq:ShiftRStep1}, and using H\"older's inequality related to the appearing $Y^m$
integral, we see that the RHS of \eqref{eq:ShiftRStep1} is bounded by
\begin{equation*}
\Big( \E \E' \int_{\R^m} \int_{Y^m} \Big |
\sum_{k \in \calI}\sum_{V \in \calD^m} \epsilon_k \epsilon_V' 1_V(x) a_{V,k}(x,y_V) \Delta_V^{j_1} f_k(y_V)
\Big |_E^q \ud \nu_m (y) \ud x \Big)^{1/q}.
\end{equation*}

Fix for the moment two points $x \in \R^m$ and $y \in Y^m$. 
The $\calR$-boundedness assumption \eqref{eq:ShiftRAs} together with
Pisier's property $(\alpha)$ show that
\begin{equation*}
\begin{split}
\E \E' \Big |
&\sum_{k \in \calI}\sum_{V \in \calD^m} \epsilon_k \epsilon_V' 1_V(x) a_{V,k}(x,y_V) \Delta_V^{j_1} f_k(y_V)
\Big |_E^q \\
& \lesssim \E''\Big |
\sum_{k \in \calI}\sum_{V \in \calD^m} \epsilon_{V,k}'' 1_V(x) a_{V,k}(x,y_V) \Delta_V^{j_1} f_k(y_V)
\Big |_E^q \\
& \lesssim C_a^q \E''\Big |
\sum_{k \in \calI}\sum_{V \in \calD^m} \epsilon_{V,k}''  1_V(x) \Delta_V^{j_1} f_k(y_V)
\Big |_E^q.
\end{split}
\end{equation*}

We have shown that
\begin{equation}\label{eq:DecLeft}
\begin{split}
\E\Big \| &\sum_{k \in \calI} \epsilon_k S^{j_1,j_2}_{\calD^m, k} f_k \Big \|_{L^q(\R^m;E)} \\
&\lesssim C_a \Big( \E''\int_{\R^m} \int_{Y^m} \Big |
\sum_{k \in \calI}\sum_{V \in \calD^m} \epsilon_{V,k}'' 1_V(x) \Delta_V^{j_1} f_k(y_V)
\Big |_E^q \ud \nu_m (y) \ud x \Big)^{1/q} \\
&\lesssim C_a\Big( \E \E'\int_{\R^m} \int_{Y^m} \Big |
\sum_{k \in \calI}\sum_{V \in \calD^m} \epsilon_k \epsilon_V' 1_V(x) \Delta_V^{j_1} f_k(y_V)
\Big |_E^q \ud \nu_m (y) \ud x\Big)^{1/q},
\end{split}
\end{equation}
where in the last step we again applied the property $(\alpha)$. Let $k \in \{0, \dots j_1\}$ and let
$\calD^m_{j_1,k} \subset \calD^m$ be the related collection as defined in \eqref{eq:SubLattice}.
The decoupling estimate \eqref{eq:decoupling} gives that
\begin{equation}\label{eq:Decouple}
\begin{split}
\E'\int_{\R^m} \int_{Y^m}
& \Big | \sum_{k \in \calI}\sum_{V \in \calD^m_{j_1,k}} \epsilon_k \epsilon_V' 1_V(x) \Delta_V^{j_1} f_k(y_V)
\Big |_E^q \ud \nu_m (y) \ud x \\
&=\E'\int_{\R^m} \int_{Y^m}
 \Big | \sum_{V \in \calD^m_{j_1,k}}  \epsilon'_V 1_V(x) \Delta^{j_1}_V\Big( \sum_{k \in \calI} \epsilon_k   f_k\Big)(y_V)
\Big |_E^q \ud \nu_m (y) \ud x \\
&\sim \int_{\R^m} 
 \Big |  \sum_{V \in \calD^m_{j_1,k}}   \Delta^{j_1}_V \Big(\sum_{k \in \calI}  \epsilon_k   f_k\Big)
\Big |_E^q \ud x 
\lesssim \int_{\R^m} \Big | \sum_{k \in \calI} \epsilon_k  f_k
 \Big |_E^q   \ud x.
\end{split}
\end{equation}
This combined with \eqref{eq:DecLeft} shows that
\begin{align*}
\E\Big \| \sum_{k \in \calI} \epsilon_k S^{j_1,j_2}_{\calD^m, k} f_k \Big \|_{L^q(\R^m;E)}
&\lesssim (j_1+1)C_a \Big(\E\Big \| \sum_{k \in \calI} \epsilon_k f_k \Big \|_{L^q(\R^m;E)}^q\Big)^{1/q} \\
&\sim (j_1+1)C_a \E\Big \| \sum_{k \in \calI} \epsilon_k f_k \Big \|_{L^q(\R^m;E)}.
\end{align*}

So far we have proved the claim with the constant $(j_1+1)$. 
The constant $(\min(j_1,j_2)+1)$ is achieved via duality.
Consider some shift $S^{j_1,j_2}_{\calD^m,k}$. Its adjoint $(S^{j_1,j_2}_{\calD^m,k})^*$ is the operator acting on functions $g \in L^{q'}(\R^m; E^*)$ by
$$
(S^{j_1,j_2}_{\calD^m,k})^*g
= \sum_{V \in \calD^m} \Delta^{j_1}_V A^*_{V,k} \Delta^{j_2}_V g,
$$
where each $A_{V,k}^*$ is an integral operator 
$$
A^*_{V,k} \varphi(y)= \frac{1_V(y)}{|V|} \int_V a_{V,k}(x,y)^* \varphi(x) \ud x, \quad \varphi \in L^1_{\text{loc}}(\R^m;E^*).
$$

Since $E$ is a UMD space, we know that if $\calT \subset \calL(E)$ is an $\calR$-bounded operator family, then
the family $\calT^*:= \{T^* \in \calL(E^*) \colon T \in \calT\}$ is also $\calR$-bounded and 
$\calR(\calT^*) \lesssim \calR(\calT)$. This can be seen using Lemma \ref{lem:randomduality}.
Thus, 
\begin{equation}\label{eq:AdjKern}
\begin{split}
\calR(\{a_{V,k}(x, y)^* \in \calL(E^* )\colon k \in \calK, V \in \calD^m, x, y \in V\}) 
\lesssim C_a.
\end{split}
\end{equation}
Also, because $E$ is a UMD space, Pisier's property $(\alpha)$ of $E$ implies that also $E^*$ has the property $(\alpha)$ with comparable constants.
See Proposition 7.5.15 in the book \cite{HNVW2}.

Hence, we see that  $\{(S^{j_1,j_2}_{\calD^m,k})^*\}_{k \in \calK}$ is a family of dyadic shifts with parameters $(j_2,j_1)$,
and the related family of kernels satisfies the $\calR$-boundedness condition \eqref{eq:AdjKern}.
The above proof shows that 
$$
\calR\big(\{ (S^{j_1,j_2}_{\calD^m,k})^* \in \calL(L^{q'}(\R^m ; E^*)) \colon k \in \calK\}) \lesssim (j_2+1)C_a.
$$
Using Lemma \ref{lem:randomduality} again we have
$$
\calR\big(\{ S^{j_1,j_2}_{\calD^m,k} \in \calL(L^{q}(\R^m ; E)) \colon k \in \calK\}) \lesssim (j_2+1)C_a.
$$
This concludes the proof.
\end{proof}

Next, we investigate shifts $S^{i_1,i_2}_{\calD^n}$ related to  families of kernels 
$a_K \colon \R^n \times \R^n \to L^p(\R^m;E)$, where $E$ is a UMD space with the property $(\alpha)$ of Pisier and $p \in (1,\infty)$ is fixed.
This time we are interested in estimates of the form
$$
 \| S^{i_1,i_2}_{\calD^n} f \|_{L^p(\R^m; L^q(\R^n;E))}
 \lesssim \| f \|_{L^p(\R^m; L^q(\R^n;E))}
 $$
for a given $q \in (1,\infty)$.
Notice that if we had the norm of $L^q(\R^n; L^p(\R^m;E))$ instead, we could apply Lemma \ref{lem:ShiftRBound}.

Let $Y^n$ be the probability space related to decoupling in $\R^n$, and suppose $T \in \calL (L^p(\R^m;E))$.
If $f \colon \R^n \times \R^m \times Y^n \to E$ is a finite sum
\begin{equation}\label{eq:Simplef}
f(x_1,x_2,y)= \sum_i 1_{A_i}(x_2) 1_{B_i}(x_1,y)e_i,
\end{equation}
where $A_i \subset \R^m$, $B_i \subset \R^n \times Y^n$ are sets of finite measure and $e_i \in E$,
then we define
$$
Tf(x_1,x_2,y):= \sum_i T(1_{A_i}e_i)(x_2) 1_{B_i}(x_1,y)e_i.
$$
The function $Tf$ is well defined i.e. independent of the representation of $f$. We say that
$T$ can be extended to an operator in $\calL(L^p(\R^m;L^q(\R^n\times Y^n;E)))$
if there exists $\tilde{T} \in \calL(L^p(\R^m;L^q(\R^n\times Y^n;E)))$ so that
$\tilde{T}f=Tf$ for all $f$ of the form \eqref{eq:Simplef}. This extension, if it exists, is unique
since functions as in \eqref{eq:Simplef} are dense in $L^p(\R^m;L^q(\R^n\times Y^n;E))$.

\begin{lem}\label{lem:ShiftWrongOrder}
Suppose $E$ is a UMD space with Pisier's property $(\alpha)$. 
Let $p,q \in (1, \infty)$ and $i_1, i_2 \in \{0,1, \dots\}$ be fixed.
Assume that $\{a_{K,k} \}_{K \in \calD^n, k \in \calK}$ is a family of kernels 
$$
a_{K,k} \colon \R^n \times \R^n \to \calL (L^p(\R^m;E)),
$$
so that each $a_{K,k}(x_1,y_1) \in \calL (L^p(\R^m;E))$ can be extended to an operator in
$$
\calL (L^p(\R^m;L^q(\R^n\times Y^n;E))).
$$ 
In addition, it is assumed that the kernels are of the form
\begin{equation}\label{eq:KernelSimple}
a_{K,k}(x_1,y_1)=\sum_{l \in \calJ_{K,k}} a_{K,k, l}1_{S_{K,k,l}}(x_1, y_1),
\end{equation}
where $(S_{K,k,l})_{l \in \calJ_{K,k}}$ is a finite partition of $K \times K$ and $a_{K, k, l} \in \calL(L^p(\R^m;E))$
Suppose that
there exists a constant $C_a$ so that 
$$
\calR(\{ a_{K,k}(x_1,y_1) \in \calL (L^p(\R^m;L^q(\R^n\times Y^n;E))) \colon k \in \calK,K \in \calD^n, x_1,y_1 \in K\}) \le C_a.
$$

For every $k \in \calK$, let  $S^{i_1,i_2}_{\calD^n,k}$ be the operator-valued dyadic shift related to the family 
$\{a_{K,k}\}_{K \in \calD^n}$. Then, every $S^{i_1,i_2}_{\calD^n,k}$ can be extended to an operator in $\calL(L^p(\R^m; L^q(\R^n;E)))$, and
$$
\calR(\{ S^{i_1,i_2}_{\calD^n,k} \in \calL (L^p(\R^m;L^q(\R^n;E))) \colon k \in \calK\}) 
\lesssim C_a (\min(i_1, i_2)+1).
$$
\end{lem}
\begin{rem}\label{rem:stronger}
The assumptions are \emph{stronger} than in Lemma \ref{lem:ShiftRBound} in the sense that they imply that
$$
\calR(\{a_{K,k}(x_1, y_1) \in \calL(L^p(\R^m; E))\colon k \in \calK, K \in \calD^n, x_1, y_1 \in K\}) \le C_a.
$$
Therefore, we have by Lemma \ref{lem:ShiftRBound} that for all $s \in (1,\infty)$ there holds
$$
\calR(\{ S^{i_1,i_2}_{\calD^n,k} \in \calL (L^s(\R^n;L^p(\R^m;E))) \colon k \in \calK\}) 
\lesssim C_a (\min(i_1, i_2)+1).
$$

The assumption \eqref{eq:KernelSimple}  is satisfied in all the applications of this lemma below.
\end{rem}

\begin{proof}[Proof of Lemma \ref{lem:ShiftWrongOrder}]
Let $\{f_k\}_{k \in \calI}$, where $\calI \subset \calK$ is finite, be a sequence of functions $f_k \colon \R^n \times \R^m \to E$ of the form
\begin{equation}\label{eq:simple}
f_k(x_1,x_2)= \sum_i 1_{A_{k,i}}(x_1)1_{B_{k,i}}(x_2)e_{k,i},
\end{equation}
where the sum is finite, $A_{k,i} \subset \R^n$ and $B_{k,i} \subset \R^m$ are sets of finite measure, and $e_{k,i} \in E$. 
By the Remark \ref{rem:stronger}, $S^{i_1,i_2}_{\calD^n,k} f_k$ is well defined for every $k$. We will show that
$$
\E \Big\| \sum_{k \in \calI}  \varepsilon_k S^{i_1,i_2}_{\calD^n,k} f_k \Big \|_{L^p(\R^m; L^q(\R^n;E))} 
\lesssim C_a (i_1+1)
\E \big\| \sum_{k \in \calI}  \varepsilon_k f_k \big \|_{L^p(\R^m; L^q(\R^n;E))},
$$
which proves Lemma \ref{lem:ShiftWrongOrder} (the minimum can be attained using duality as before).

Below we view the functions $f_k$ as  functions in $L^q(\R^n;L^p(\R^m;E))$, so that the martingale differences 
$\Delta^{i_1}_Kf_k$ have the usual meaning as  $L^p(\R^m;E)$-valued functions.
Begin by estimating (operate in $L^q(\R^n;E)$ with a fixed $x_2 \in \R^m$ to introduce random signs and to get rid
of the martingales, and use Kahane--Khintchine):
\begin{equation}\label{eq:ShiftWrong1}
\begin{split}
\E \Big\| \sum_{k \in \calI} & \epsilon_k S^{i_1,i_2}_{\calD^n,k} f_k \Big\|_{L^p(\R^m; L^q(\R^n;E))} \\
&=\E \Big\| \sum_{K \in \calD^n} \Delta^{i_2}_{K} \sum_{k \in \calI} \epsilon_k A_{K,k} \Delta^{i_1}_{K} f_k \Big\|_{L^p(\R^m; L^q(\R^n;E))} \\
&\lesssim \E \E' \Big\| \sum_{K \in \calD^n}  \sum_{k \in \calI} \epsilon'_K\epsilon_k A_{K,k} \Delta^{i_1}_{K} f_k \Big\|_{L^p(\R^m; L^q(\R^n;E))} \\
& \sim \E \Big\| \sum_{K \in \calD^n}  \sum_{k \in \calI} \epsilon_{K,k} A_{K,k} \Delta^{i_1}_{K} f_k \Big\|_{L^p(\R^m; L^q(\R^n;E))}.
\end{split}
\end{equation}
The last step applied the property $(\alpha)$ of $L^p(\R^m; L^q(\R^n;E))$.

As in Lemma \ref{lem:ShiftRBound}, we write
\begin{equation*}
\begin{split}
A_{K,k}\Delta^{i_1}_{K}f_k(x_1)&= \frac{1_K(x_1)}{|K|} \int_K a_K(x_1,y_1) \Delta^{i_1}_{K}f_k(y_1) \ud y_1 \\
& = \int_{Y^n}1_K(x_1)  a_{K,k}(x_1,y_K) \Delta^{i_1}_{K}f_k(y_K) \ud \nu_n (y).
\end{split}
\end{equation*}
The interpretation here is that $\Delta^{i_1}_{K}f_k(y_K) \in L^p(\R^m;E)$, to which
$a_{K,k}(x_1,y_K) \in \calL(L^p(\R^m;E))$ hits giving $a_{K,k}(x_1,y_K) \Delta^{i_1}_{K}f_k(y_K) \in L^p(\R^m;E)$.
This can further be evaluated at $x_2 \in \R^m$ to get an element of $E$. This will simply be written
as $a_{K,k}(x_1,y_K) \Delta^{i_1}_{K}f_k(y_K)(x_2) \in E$.

We can assume that the kernels $A_{K,k}$ are supported in $K \times K$, 
and so we can stop writing the indicator $1_K(x_1)$.
Thus,  the right hand side of \eqref{eq:ShiftWrong1} is dominated by
\begin{equation}\label{eq:WrongY}
 \E   \Big\|  \sum_{K \in \calD^n}  \sum_{k \in \calI} \epsilon_{K,k} 
a_{K,k}(x_1,y_K) \Delta^{i_1}_{K}f_k(y_K) (x_2)\Big\|_{L^p(\ud x_2; L^q (\ud x_1 \times \nu_n(y);E))}.
\end{equation}

To proceed, we aim to apply the fact that the kernels $a_{K,k}$ are of the form \eqref{eq:KernelSimple}.
Kahane--Khintchine inequality implies that \eqref{eq:WrongY} is equivalent with
$$
\Big \| \Big( \int_{\R^n} \int_{Y^n} \E \Big| \sum_{K \in \calD^n}  \sum_{k \in \calI} \epsilon_{K,k} 
a_{K,k}(x_1,y_K)\Delta^{i_1}_{K}f_k(y_K) (x_2)\Big|^q_E \ud \nu_n (y) \ud x_1 \Big)^{1/q} \Big \|_{L^p(\ud x_2)}.
$$
Fix $x_1 \in \R^n$ and $y \in Y^n$. Let 
$$
\{\epsilon_{K, k, l}\colon K \in \calD^n, k \in \calK, l \in \calJ_{K,k}\}
$$  
be another independent sequence of random signs. By the identical distribution of
$$
\{\epsilon_{K,k}\colon K \in \calD^n, k \in \calK, x_1 \in K\}
$$
and
$$
\{\epsilon_{K,k,l}\colon K \in \calD^n, k \in \calK, l \in \calJ_{K,k} \textup{ s.t. } (x_1, y_K) \in S_{K,k,l}\}
$$
we have
\begin{equation}\label{eq:IDTrick}
\begin{split}
\E &\Big| \sum_{K \in \calD^n}  \sum_{k \in \calI} \epsilon_{K,k} 
a_{K,k}(x_1,y_K) \Delta^{i_1}_{K}f(y_K) (x_2)\Big|^q_E \\
& =\E \Big| \sum_{k \in \calI} \sum_{K \in \calD^n} \epsilon_{K,k}  \sum_{l \in \calJ_{K,k}} 
1_{S_{K,k,l}}(x_1, y_K)  a_{K,k, l}
 \Delta^{i_1}_{K}f_k(y_K)(x_2)\Big|^q_E  \\
 &=\E \Big| \sum_{k \in \calI} \sum_{K \in \calD^n}\sum_{l \in \calJ_{K,k}}  \epsilon_{K,k,l}  
1_{S_{K,k,l}}(x_1, y_K)  a_{K,k, l}
 \Delta^{i_1}_{K}f_k(y_K)(x_2)\Big|^q_E.
\end{split}
\end{equation}

Using \eqref{eq:IDTrick} and applying the Kahane--Khintchine inequality again we have shown that \eqref{eq:WrongY}
is equivalent with
\begin{equation}\label{eq:ReadyForR}
 \E   \Big \|    \sum_{k \in \calI} \sum_{K \in \calD^n}\sum_{l \in \calJ_{K,k}}  \epsilon_{K,k,l}  
1_{S_{K,k,l}}(x_1, y_K)  a_{K,k, l}
 \Delta^{i_1}_{K}f_k(y_K)(x_2) \Big \|,
\end{equation}
where the norm $ \| \cdot\|$ refers to $\| \cdot \|_{L^p(\ud x_2; L^q (\ud x_1 \times \nu_n(y);E))}$.
Define for every $(K, k, l)$ the function
$$
1_{S_{K,k,l}}(x_1, y_K)  \Delta^{i_1}_{K}f_k(y_K) (x_2)=: F_{K,k,l}(x_1,x_2,y).
$$
Using the fact that $f$ is of the form \eqref{eq:simple} we see that
$$
1_{S_{K,k,l}}(x_1, y_K)  a_{K,k,l}
\Delta^{i_1}_{K}f_k(y_K)(x_2)
=a_{K,k,l} F_{K,k,l}(x_1,x_2,y),
$$
where in the right hand side we interpreted $a_{K,k,l}$ as the extended operator in $\calL (L^p(\R^m;L^q(\R^n\times Y^n;E)))$.
Now, the assumed $\calR$-boundedness gives that
\begin{equation*}
\begin{split}
\eqref{eq:ReadyForR}
&=\E   \Big \|   \sum_{k \in \calI} \sum_{K \in \calD^n}\sum_{l \in \calJ_{K,k}} 
\epsilon_{K,k,l}  a_{K,k,l} F_{K,k,l}(x_1,x_2,y) \Big \| \\
& \lesssim C_a \E   \Big \|   \sum_{k \in \calI} \sum_{K \in \calD^n}\sum_{l \in \calJ_{K,k}} 
\epsilon_{K,k,l} 1_{S_{K,k,l}}(x_1, y_K)  \Delta^{i_1}_{K}f_k(y_K)(x_2) \Big \| \\
& \sim C_a \E   \Big \|    \sum_{k \in \calI}  \sum_{K \in \calD^n}
\epsilon_{K,k}  1_{K}(x_1) 
\Delta^{i_1}_{K}f_k(y_K)(x_2) \Big \|,
\end{split}
\end{equation*} 
where in the last step we converted back to the random signs $\epsilon_{K,k}$ as in \eqref{eq:IDTrick}.

Finally, using the property $(\alpha)$ of $L^p(\ud x_2; L^q (\ud x_1 \times \nu_n(y);E))$ and then decoupling similarly
as in \eqref{eq:Decouple} it is seen that
\begin{equation*}
\begin{split}
\E   \Big \|    \sum_{k \in \calI}  &\sum_{K \in \calD^n}
\epsilon_{K,k}  1_{K}(x_1) 
\Delta^{i_1}_{K}f_k(y_K)( x_2) \Big \|_{L^p(\ud x_2; L^q (\ud x_1 \times \nu_n(y);E))} \\
&\sim \E \E'  \Big \| \sum_{K \in \calD^n} \epsilon_K' 1_K(x_1)\Delta^{i_1}_{K} 
\Big( \sum_{k \in \calI} \epsilon_k 
f_k\Big)(y_K)( x_2) \Big \|_{L^p(\ud x_2; L^q (\ud x_1 \times \nu_n(y);E))} \\
&\lesssim (i_1+1) \E \Big\| \sum_{k \in \calI}  \epsilon_k f_k \Big \|_{L^p(\R^m; L^q(\R^n;E))}
\end{split}
\end{equation*}
This concludes the proof.

\end{proof}

\subsection{Bi-parameter operator-valued shifts}
We turn to show that operator-valued bi-parameter shifts are bounded (even $\calR$-bounded as a family).
\begin{prop}\label{prop:biparOP}
Let $E$ be a UMD space satisfying the property $(\alpha)$ of Pisier, and let $i_1, i_2, j_1, j_2 \ge 0$ be fixed parameters.
Suppose $\{S^{i_1, i_2, j_1,j_2}_{\calD^n, \calD^m,k}\}_{k \in \calK}$ is a family of operator-valued bi-parameter dyadic shifts
as in Section \ref{ss:opbipardef}. For every $k \in \calK$ let
$$
\{a_{K, V,k} \colon K \in \calD^n, V \in \calD^m\}
$$
be the family of kernels related to the shift $S^{i_1, i_2, j_1, j_2}_{\calD^n, \calD^m,k}$.
Assume that there exists a constant $C_a$ so that
\begin{equation*}
\calR\big(\{ a_{K,V,k}(x,y) \colon k \in \calK, K \in \calD^n, V \in \calD^m, x,y \in K\times V \}\big) \le C_a.
\end{equation*}
Then for all $p, q \in (1,\infty)$ we have
\begin{align*}
\calR\big(\{ S^{i_1,i_2,j_1,j_2}_{\calD^n, \calD^m, k} \in \calL(L^q(\R^n&; L^p(\R^m;E))) \colon k \in \calK\}) \\
&\lesssim (\min(i_1, i_2)+1)(\min(j_1, j_2)+1)C_a. 
\end{align*}
\end{prop}
\begin{proof}
For each fixed $k \in \calK$, $K \in \calD^n$ and $x_1, y_1 \in K$ we define the one-parameter operator-valued dyadic shift in $\R^m$ by the formula
$$
S^{j_1, j_2}_{\calD^m, k, K, x_1, y_1} \varphi := \sum_{V \in \calD^m} \Delta_V^{j_2} A_{V, k}^{K, x_1, y_1} \Delta_V^{j_1} \varphi, \qquad \varphi \in L_{\loc}^1(\R^m;E),
$$
where
$$
A_{V, k}^{K, x_1, y_1} \varphi(x_2) := \frac{1_V(x_2)}{|V|} \int_V a_{K,V, k}(x_1, x_2, y_1, y_2) \varphi(y_2)\ud y_2.
$$
The assumptions and Lemma \ref{lem:ShiftRBound} show that for all $p \in (1,\infty)$ we have
\begin{equation}\label{eq:r1}
\begin{split}
\calR\big(\{ S^{j_1,j_2}_{\calD^m, k, K, x_1, y_1} \in & \calL(L^p(\R^m ; E)) \colon  k \in \calK, K \in \calD^m, x_1, y_1 \in K\}) \\
& \lesssim (\min(j_1,j_2)+1)C_a.
\end{split}
\end{equation}

Next, fix $p \in (1,\infty)$, and for each $k \in \calK$ define the one-parameter operator-valued dyadic shift in $\R^n$ by the formula
$$
S^{i_1, i_2}_{\calD^n, k} \psi := \sum_{K \in \calD^n} \Delta_K^{i_2} A_{K, k} \Delta_K^{i_1} \psi, \qquad \psi \in L^1_{\loc}(\R^n; L^p(\R^m; E)),
$$
where
$$
A_{K,k} \psi(x_1) = \frac{1_K(x_1)}{|K|} \int_K a_{K,k}(x_1,y_1)\psi(y_1)\ud y_1
$$
and
$$
a_{K,k}(x_1, y_1) = S^{j_1, j_2}_{\calD^m, k, K, x_1, y_1}.
$$
Recall that $L^p(\R^m; E)$ is a UMD space with the property $(\alpha)$ of Pisier.
Using \eqref{eq:r1} and Lemma \ref{lem:ShiftRBound} we see that for all $q \in (1,\infty)$ there holds
\begin{align*}
\calR\big(\{ S^{i_1, i_2}_{\calD^n, k}\in & \calL(L^q(\R^n; L^p(\R^m; E))) \colon  k \in \calK\})
\lesssim (\min(i_1,i_2)+1)(\min(j_1,j_2)+1)C_a.
\end{align*}

To conclude the proof, we only need to check that
$$
S^{i_1, i_2}_{\calD^n, k} f = S^{i_1, i_2, j_1, j_2}_{\calD^n, \calD^m, k} f.
$$
For this identity we consider $k \in \calK$ fixed, and suppress it from the notation.
A straightforward, however tedious, way is to expand both sides using Haar functions. Clearly, $S^{i_1, i_2, j_1, j_2}_{\calD^n, \calD^m} f$ equals
\begin{align*}
\mathop{\sum_{K \in \calD^n}}_{V \in \calD^m} \frac{1}{|K||V|} \mathop{\sum_{I_1, I_2 \in \calD^n}}_{I_1^{(i_1)} = I_2^{(i_2) }=K}
\mathop{\sum_{J_1, J_2 \in \calD^m}}_{J_1^{(j_1)} = J_2^{(j_2) }=V} \Big( &\iint_{I_1 \times J_1} \iint_{I_2 \times J_2} 
(h_{I_1} \otimes h_{J_1})(y)(h_{I_2}\otimes h_{J_2})(z) \\
&\times a_{K,V}(z,y) \langle f, h_{I_1} \otimes h_{J_1} \rangle\ud z \ud y \Big) h_{I_2}\otimes h_{J_2}.
\end{align*}

We now check that also $S^{i_1, i_2}_{\calD^n} f$ equals this. Since
$$
\Delta_K^{i_1} f(y_1) = \mathop{\sum_{I_1\in \calD^n}}_{I_1^{(i_1)} =K} h_{I_1}(y_1) \langle f, h_{I_1} \rangle_1,
$$
we have
$$
A_K \Delta_K^{i_1} f(z_1)  
= \frac{1_K(z_1)}{|K|} \mathop{\sum_{I_1\in \calD^n}}_{I_1^{(i_1)} =K} \int_{I_1} h_{I_1}(y_1) S^{j_1, j_2}_{\calD^m, K, z_1, y_1} \langle f, h_{I_1} \rangle_1 \ud y_1
$$
and
\begin{align*}
\Delta_K^{i_2} A_{K} \Delta_K^{i_1} f(x_1) = \frac{1}{|K|} \mathop{\sum_{I_1, I_2 \in \calD^n}}_{I_1^{(i_1)} = I_2^{(i_2) }=K} h_{I_2}(x_1)
\Big( \int_{I_1} &\int_{I_2} h_{I_1}(y_1) h_{I_2}(z_1) \\
&\times S^{j_1, j_2}_{\calD^m, K, z_1, y_1} \langle f, h_{I_1} \rangle_1 \ud z_1  \ud y_1 \Big) .
\end{align*}
Since
$$
\Delta_V^{j_1} \langle f, h_{I_1} \rangle_1(y_2) = \mathop{\sum_{J_1 \in \calD^m}}_{J_1^{(j_1)} =V} \langle f, h_{I_1} \otimes h_{J_1} \rangle h_{J_1}(y_2),
$$
we have
\begin{align*}
A_V^{K, z_1, y_1}\Delta_V^{j_1} \langle f, h_{I_1} \rangle_1 (z_2) = \frac{1_V(z_2)}{|V|} \mathop{\sum_{J_1 \in \calD^m}}_{J_1^{(j_1)} =V} &
\int_{J_1} h_{J_1}(y_2) \\
&\times a_{K,V}(z_1, z_2, y_1, y_2) \langle f, h_{I_1} \otimes h_{J_1} \rangle\ud y_2
\end{align*}
and
\begin{align*}
\Delta_V^{j_2} A_{V}^{K, z_1, y_1} \Delta_V^{j_1} \langle f, h_{I_1} \rangle_1(x_2)
= \frac{1}{|V|}&\mathop{\sum_{J_1, J_2 \in \calD^m}}_{J_1^{(j_1)} = J_2^{(j_2) }=V} 
 \Big( \int_{J_1} \int_{J_2} h_{J_1}(y_2) h_{J_2}(z_2)  \\
&\times a_{K,V}(z_1, z_2, y_1, y_2) \langle f, h_{I_1} \otimes h_{J_1} \rangle\ud z_2 \ud y_2\Big)h_{J_2}(x_2).
\end{align*}
Combining, we readily see that $S^{i_1, i_2}_{\calD^n} f$ has the same Haar expansion as
$S^{i_1, i_2, j_1, j_2}_{\calD^n, \calD^m} f$, and the proof is complete.
\end{proof}

\section{Model operators}\label{sec:model}
Fix a UMD space $F$.
Suppose that for every $K, I_1, I_2 \in \calD^n$ we are given an operator $B_{K, I_1, I_2} \in \calL(F)$. Fix two indices $i_1, i_2 \ge 0$, and define the model operator
$$
P^{i_1, i_2}_{\calD^n} f(x) = \sum_{K \in \calD^n} \mathop{\sum_{I_1, I_2 \in \calD^n}}_{I_1^{(i_1)} = I_2^{(i_2)} = K} h_{I_2}(x) B_{K, I_1, I_2}( \langle f, h_{I_1} \rangle),
$$
where $x \in \R^n$ and $f \colon \R^n \to F$ is locally integrable. The next proposition considers
a family of these operators $P^{i_1, i_2}_{\calD^n,b}$, where $b \in \mathcal{B}$ and $\mathcal{B}$ is some index set.

\begin{prop}\label{prop:mod1-1}
Let $F$ be a UMD space with the property $(\alpha)$ of Pisier.
Suppose that
$$
\calR\Big(\Big\{ \frac{|K|}{|I_1|^{1/2} |I_2|^{1/2}} B_{K, I_1, I_2, b} \in \calL(F)\colon K, I_1, I_2 \in \calD^n, 
b \in \mathcal{B}\Big\} \Big) \le C_0.
$$
Let $P_{\calD^n, b}^{i_1,i_2}$ be a model operator associated with the operators $B_{K, I_1, I_2, b}$.
Then for all $q \in (1,\infty)$ we have
$$
\calR(\{P_{\calD^n, b}^{i_1,i_2} \in \calL(L^q(\R^n; F))\colon b \in \mathcal{B}\}) \lesssim (\min(i_1, i_2)+1)C_0.
$$
\end{prop}
\begin{proof}
This is essentially just an estimate for operator-valued shifts in a form which is a priori slightly different.
To see the simple connection define the operator-valued kernels
$$
a_{K,b}^{i_1, i_2} \colon \R^n \times \R^n \to \calL(F)
$$
by setting
$$
a_{K,b}^{i_1, i_2}(x,y) = |K| \mathop{\sum_{I_1, I_2 \in \calD^n}}_{I_1^{(i_1)} = I_2^{(i_2)} = K} h_{I_1}(y) h_{I_2}(x) B_{K, I_1, I_2, b}.
$$
We then define the averaging operator $A_{K,b}^{i_1, i_2}$, mapping locally integrable functions $f \colon \R^n \to F$ to
$A_{K,b}^{i_1, i_2}f \colon \R^n \to F$, by the formula
$$
A_{K,b}^{i_1, i_2}f(x) = \frac{1}{|K|} \int_{\R^n} a_{K,b}^{i_1,i_2}(x, y) f(y) \ud y. 
$$
Define the operator-valued shift
$$
S_{\calD^n, b}^{i_1, i_2}f  := \sum_{K \in \calD^n} A_{K,b}^{i_1, i_2} f  = \sum_{K \in \calD^n} \Delta_K^{i_2} A_{K,b}^{i_1, i_2} \Delta_K^{i_1} f.
$$
By Lemma \ref{lem:ShiftRBound} we have for all $q \in (1,\infty)$ that
$$
\calR(\{S_{\calD^n, b}^{i_1,i_2} \in \calL(L^q(\R^n; F))\colon b \in \mathcal{B}\}) \lesssim (\min(i_1, i_2)+1)R,
$$
where
$$
R := \calR(\{a_{K,b}^{i_1, i_2}(x,y) \in \calL(F)\colon b \in \mathcal{B}, K \in \calD^n, x, y \in K\}).
$$
Consider a fixed tuple $(b, K, x, y)$ so that $a_{K,b}^{i_1, i_2}(x, y) \ne 0$. Then there are
unique $I_1, I_2 \in \calD^n$ (depending on $(K, x, y)$) so that $I_1^{(i_1)} = I_2^{(i_2)} = K$, $y \in I_1$ and $x \in I_2$.
Now, we have
$$
a_{K,b}^{i_1, i_2}(x, y) = |K| h_{I_1}(y) h_{I_2}(x) B_{K, I_1, I_2, b} = \pm \frac{|K|}{|I_1|^{1/2} |I_2|^{1/2}} B_{K, I_1, I_2,b}.
$$
Thus, we have $R \le C_0$. It remains only to notice that $S_{\calD^n, b}^{i_1, i_2}f = P_{\calD^n, b}^{i_1,i_2}f$, which follows from the fact that
$$
A_{K, b}^{i_1, i_2}f(x) = \mathop{\sum_{I_1, I_2 \in \calD^n}}_{I_1^{(i_1)} = I_2^{(i_2)} = K} h_{I_2}(x) B_{K, I_1, I_2, b} \Big( \int_{I_1} f(y)h_{I_1}(y)\ud y\Big).
$$
\end{proof}
Let us now consider the special case of model operators, where $F = L^p(\R^m; E)$ for some
fixed $p \in (1,\infty)$ and UMD space $E$. We formulate a condition for verifying the boundedness
of $P^{i_1,i_2}_{\calD^n}$ from $L^p(\R^m; L^q(\R^n; E))$ to $L^p(\R^m; L^q(\R^n; E))$. Notice that
the previous proposition only allows to conclude that under certain conditions
$P^{i_1,i_2}_{\calD^n}$ is bounded from $L^q(\R^n; L^p(\R^m; E)) \to L^q(\R^n; L^p(\R^m; E))$ for all $q \in (1,\infty)$.
The condition will now depend both on $p$ and $q$. These assumptions are stronger than above (i.e.
they also imply the conclusion of Proposition \ref{prop:mod1-1}), see Remark \ref{rem:stronger}.

When we talk about extensions, we always mean tensor extensions as in Section \ref{sec:biopshift}.
\begin{prop}\label{prop:mod1-2}
Let $E$ be a UMD space with the property $(\alpha)$ of Pisier, $p,q \in (1,\infty)$ and $F = L^p(\R^m; E)$.
Suppose we are given operators $B_{K, I_1, I_2, b} \in \calL(F)$ that can be extended
to operators in $\calL(L^p(\R^m; L^q(\R^n \times Y^n; E)))$, and that
\begin{align*}
\calR\Big(\Big\{ \frac{|K|}{|I_1|^{1/2} |I_2|^{1/2}} B_{K, I_1, I_2, b} \in \calL(L^p(\R^m; &L^q(\R^n \times Y^n; E)))\colon \\
& K, I_1, I_2 \in \calD^n, b \in \mathcal{B}\Big\} \Big) \le C_0.
\end{align*}
Let $P_{\calD^n, b}^{i_1,i_2}$ be a model operator associated with the operators $B_{K, I_1, I_2, b}$.
Then we have
$$
\calR(\{P_{\calD^n, b}^{i_1,i_2} \in \calL(L^p(\R^m; L^q(\R^n; E)))\colon b \in \mathcal{B}\}) \lesssim (\min(i_1, i_2)+1)C_0.
$$
\end{prop}
\begin{proof}
As in the proof of Proposition \ref{prop:mod1-1} we have that $P_{\calD^n, b}^{i_1,i_2} = S_{\calD^n,b}^{i_1,i_2}$
for a certain operator-valued shift. Using Lemma \ref{lem:ShiftWrongOrder} we get the claim exactly as before.
\end{proof}
For the purposes of tri-parameter theory let us still go one step further. So suppose now that $F = L^p(\R^m; L^r(\R^k; E))$.
Notice that if we now consider boundedness in $L^q(\R^n; F)$ we can use Proposition \ref{prop:mod1-1}.
On the other hand, if we consider boundedness in $L^p(\R^m; L^q(\R^n; L^r(\R^k;E)))$ we need to use Proposition \ref{prop:mod1-2}.
The case $L^p(\R^m; L^r(\R^k; L^q(\R^n; E)))$ requires a new proposition.
\begin{prop}\label{prop:mod1-3}
Let $E$ be a UMD space with the property $(\alpha)$ of Pisier, $p,q, r \in (1,\infty)$ and $F = L^p(\R^m; L^r(\R^k; E))$.
Suppose we are given operators $B_{K, I_1, I_2, b} \in \calL(F)$ that can be extended
to operators in $\calL(L^p(\R^m; L^r(\R^k; L^q(\R^n \times Y^n; E))))$, and that
\begin{align*}
\calR\Big(\Big\{ \frac{|K|}{|I_1|^{1/2} |I_2|^{1/2}} B_{K, I_1, I_2, b} \in \calL(L^p(\R^m; L^r(\R^k&; L^q(\R^n \times Y^n; E))))\colon \\
 &K, I_1, I_2 \in \calD^n, b \in \mathcal{B}\Big\} \Big) \le C_0.
\end{align*}
Let $P_{\calD^n, b}^{i_1,i_2}$ be a model operator associated with the operators $B_{K, I_1, I_2, b}$.
Then we have
$$
\calR(\{P_{\calD^n, b}^{i_1,i_2} \in \calL(L^p(\R^m; L^r(\R^k; L^q(\R^n; E))))\colon b \in \mathcal{B}\}) \lesssim (\min(i_1, i_2)+1)C_0.
$$
\end{prop}
\begin{proof}
As in the proof of Proposition \ref{prop:mod1-1} we have that $P_{\calD^n, b}^{i_1,i_2} = S_{\calD^n,b}^{i_1,i_2}$
for a certain operator-valued shift. Using an obvious variant of Lemma \ref{lem:ShiftWrongOrder} (the proof is essentially the same) we get the claim.
\end{proof}

\section{$\calR$-boundedness results for paraproducts}\label{sec:RbounbforPar}
\subsection{$\calR$-boundedness of one-parameter paraproducts}
We begin by giving a nice and elementary argument showing the $\calR$-boundedness of paraproducts
$\pi_{\calD^m, b_i} \in \calL(L^p(\R^m))$ when $\|b_i\|_{\BMO_{\calD^m}} \le 1$. This proof may be of independent interest. However, it is not needed
as, using other techniques, we prove a more general result right after.

In this more general result  we study $\calR$-boundedness in $L^p(\R^m; E)$ for a UMD function lattice $E$.
Recall that UMD-valued paraproducts are bounded in $L^p$. Therefore, it seems reasonable to suspect that Proposition \ref{prop:1parprodLattice} is true in the
generality that $E$ is a UMD space satisfying Pisier's property $(\alpha)$. However, showing that would certainly require different methods.
\begin{prop}\label{prop:1parprod}
Suppose that $\|b_i\|_{\BMO_{\calD^m}} \le 1$, $i \in \calI$, and $p \in (1,\infty)$. Then
$$
\calR(\{ \pi_{\calD^m, b_i} \in \calL(L^p(\R^m))\colon i \in \calI\}) \lesssim 1.
$$
\end{prop}
\begin{proof}

This proof has the benefit that we can work with the $L^1$ definition of BMO directly, and we even don't need
to use the John--Nirenberg inequality.

We start with simple stopping time preliminaries.
Consider first a single function $b$ for which $\|b\|_{\BMO_{\calD^m}} \le 1$, and a fixed $J_0 \in \calD^m$. We set
$\calF^0_{b}(J_0) = \{J_0\}$, and let $\calF^1_{b}(J_0)$ consist of the maximal $J \in \calD^m$, $J \subset J_0$, for which $|\langle b \rangle_J - \langle b \rangle_{J_0}| > 4$.
Notice that for all $J \in \calF^1_{b}(J_0)$ we have
$$
4 < |\langle b \rangle_J - \langle b \rangle_{J_0}| \le \frac{1}{|J|} \int_J |b - \langle b \rangle_{J_0}|, 
$$
so that
$$
\Big| \bigcup_{J \in \calF^1_{b}(J_0)} J \Big| \le \frac{1}{4} \int_{J_0} |b - \langle b \rangle_{J_0}| \le \frac{|J_0|}{4}.
$$
Iterating this scheme we get the sparse family of stopping cubes defined by $\calF_b(J_0) = \bigcup_{j=0}^{\infty}\calF^j_{b}(J_0) $. 
A family of cubes is sparse if for each cube $Q$ in the family there is a subset $E_Q \subset Q$ so that $|E_Q| \gtrsim |Q|$ and so that
the sets $E_Q$ are pairwise disjoint.

For every dyadic $Q \subset J_0$ we let
$\pi_{\calF_b(J_0)} Q$ denote the minimal $J \in \calF_b(J_0)$ so that $Q \subset J$. The following estimate for martingale blocks is key to us:
$$
\Big\| \mathop{\sum_{Q \in \calD^m}}_{\pi_{\mathcal{F}_b(J_0)} Q = J} \Delta_Q b \Big\|_{L^{\infty}(\R^m)} \lesssim 1, \qquad J \in \mathcal{F}_b(J_0).
$$

Another stopping time we use is the standard principal cubes of a function $f \in L^1_{\loc}(\R^m)$. This means that
$\mathcal{S}_f^0(J_0) = \{J_0\}$, and we let $\mathcal{S}^1_{f}(J_0)$ consist of the maximal $J \in \calD^m$, $J \subset J_0$, for which $\langle |f| \rangle_J > 4\langle |f|\rangle_{J_0}$.
This time it is perhaps even more trivial that
$$
\Big| \bigcup_{J \in \mathcal{S}^1_{f}(J_0)} J \Big| \le \frac{|J_0|}{4}.
$$
Iterating this we get the sparse family of stopping cubes defined by $\mathcal{S}_f(J_0)$. 

Our final stopping time is established by combining these two in the following sense. Let $\calF^0_{b, f}(J_0) = \{J_0\}$ and let
$\calF^1_{b, f}(J_0)$ be the maximal cubes of $\calF^1_{b}(J_0) \cup \mathcal{S}^1_{f}(J_0)$. The final
sparse collection, established by iterating this, is denoted by $\calF_{b, f}(J_0) = \bigcup_{j=0}^{\infty}\calF^j_{b, f}(J_0)$.

After these preliminaries we give the actual proof. We need to show that
given a finite $\mathcal{J} \subset \calI$ and $f_j \in L^p(\R^{m})$ we have
$$
\Big\| \Big( \sum_{j \in \mathcal{J}} | \pi_{b_j} f_j |^2 \Big)^{1/2} \Big\|_{L^p(\R^{m})} \lesssim 
\Big\| \Big( \sum_{j \in \mathcal{J}} | f_j |^2 \Big)^{1/2} \Big\|_{L^p(\R^{m})},
$$
where we abbreviated $\pi_{b_j} = \pi_{\calD^m, b_j}$. 

Using calculations as in Section \ref{ss:estimates} it is clear that the following one-parameter analog of the results of that section holds:
\begin{equation}\label{eq:vvburk}
\Big\| \Big( \sum_{j \in \mathcal{J}} | f_j |^2 \Big)^{1/2} \Big\|_{L^p(\R^{m})}
\sim \Big\| \Big( \sum_{j \in \mathcal{J}} \sum_{Q \in \calD^m} |\Delta_Q f_j|^2 \Big)^{1/2} \Big\|_{L^p(\R^m)}.
\end{equation}
This is also stated in Lemma 2.1 of \cite{Vu1}.
Using this we have
\begin{equation*}
\Big\| \Big( \sum_{j \in \mathcal{J}} | \pi_{b_j} f_j |^2 \Big)^{1/2} \Big\|_{L^p(\R^{m})}
\sim \Big\| \Big( \sum_{j \in \mathcal{J}} \sum_{Q \in \calD^m} | \langle f_j \rangle_Q \Delta_Q b_j |^2 \Big)^{1/2} \Big\|_{L^p(\R^{m})}. 
\end{equation*}
Therefore, it is enough to fix an arbitrary $J_0 \in \calD^m$ and prove the estimate
\begin{equation*}
\Big\| \Big( \sum_{j \in \mathcal{J}} \sum_{\substack{Q \in \calD^m  \\Q \subset J_0}} | \langle f_j \rangle_Q \Delta_Q b_j |^2 \Big)^{1/2} \Big\|_{L^p(\R^{m})}
\lesssim \Big\| \Big( \sum_{j \in \mathcal{J}} | f_j |^2 \Big)^{1/2} \Big\|_{L^p(\R^{m})}.
\end{equation*}

For every $j \in \calJ$ we set $\calF_j := \calF_{b_j, f_j}(J_0)$.
Then, we estimate
\begin{equation}\label{eq:UnderStopping}
\begin{split}
\Big\| \Big( \sum_{j \in \mathcal{J}}  \sum_{\substack{Q \in \calD^m  \\Q \subset J_0}}&| \langle f_j \rangle_Q \Delta_Q b_j |^2 \Big)^{1/2} \Big\|_{L^p(\R^{m})} \\
& \lesssim \Big\| \Big( \sum_{j \in \mathcal{J}} \sum_{J \in  \calF_j} \langle |f_j| \rangle_J ^2
\sum_{\substack{Q \in \calD^m \\ \pi_{\calF_j}Q=J}} |  \Delta_Q b_j |^2 \Big)^{1/2} \Big\|_{L^p(\R^{m})} \\
& \sim \Big\| \Big( \sum_{j \in \mathcal{J}} \sum_{J \in  \calF_j} \langle |f_j| \rangle_J ^2
\Big|\sum_{\substack{Q \in \calD^m \\ \pi_{\calF_j}Q=J}}  \Delta_Q b_j \Big|^2 \Big)^{1/2} \Big\|_{L^p(\R^{m})}.
\end{split}
\end{equation}
The last step applied \eqref{eq:vvburk} again.
Since for every $j \in \calJ$ and $J \in \calF_j$ we have
$$
\Big|\sum_{\substack{Q \in \calD^m \\ \pi_{\calF_j}Q=J}}  \Delta_Q b_j \Big|
\lesssim 1_J,
$$
the right hand side of \eqref{eq:UnderStopping} is further bounded by
\begin{equation*}
\begin{split}
\Big\| \Big( \sum_{j \in \mathcal{J}} \sum_{J \in  \calF_j} \langle |f_j| \rangle_J ^2
1_J \Big)^{1/2} \Big\|_{L^p(\R^{m})}
\lesssim \Big\| \Big( \sum_{j \in \mathcal{J}} | f_j |^2 \Big)^{1/2} \Big\|_{L^p(\R^{m})}.
\end{split}
\end{equation*}
The last step used a version of the Carleson embedding theorem stated at least in Lemma 2.2 of \cite{Vu1}.
\end{proof}

With a different proof we can manage the generality that $E$ is a UMD function lattice. The method is very similar to those that
we use with bi-parameter paraproducts below.
\begin{prop}\label{prop:1parprodLattice}
Suppose that $\|b_i\|_{\BMO_{\calD^m}} \le 1$, $i \in \calI$, $E$ is a UMD function lattice and $p \in (1,\infty)$. Then
$$
\calR(\{ \pi_{\calD^m, b_i} \in \calL(L^p(\R^m; E))\colon i \in \calI\}) \lesssim 1.
$$
\end{prop}
\begin{proof}
Denote $\pi_{b_i} = \pi_{\calD^m, b_i}$. For all finite subsets $J \subset \calI$ and $f_j \in L^p(\R^m;E)$, $g_j \in L^{p'}(\R^m ;E^*)$ (recall that 
$E^* = E'$) we will show that
\begin{align*}
\Big|\sum_{j \in \calJ} \int_{\R^m}& \{\pi_{b_j} f_j(x), g_j(x)\}_E \ud x\Big| \\ 
&\lesssim \Big\| \Big( \sum_{j \in \calJ}  |f_j|^2 \Big)^{1/2} \Big\|_{L^p(\R^m ;E)}
\Big\| \Big( \sum_{j \in \calJ}  |g_j|^2 \Big)^{1/2} \Big\|_{L^{p'}(\R^m ;E^*)}.
\end{align*}
This is enough as can be seen by using Lemma \ref{lem:randomduality} and
Lemma \ref{lem:KKlemma}. As $g_j(x) \in E^* = E'$ the pairings $\{\cdot, \cdot\}_E$ are just integrals
(over some space $\Omega$ appearing in the definition of function lattices, see Section \ref{ss:fl}). This is convenient for checking the validity of some of the manipulations below.

Recalling that
$$
\pi_{b_j} f_j(x) = \sum_{J \in \calD^m} \langle b_j, h_J\rangle \langle f_j \rangle_J h_J(x), \qquad x \in \R^m,
$$
we get
\begin{align*}
\Big| \sum_{j \in \calJ} \int_{\R^m} \{\pi_{b_j} f_j(x), g_j(x)\}_E \ud x \Big|
&= \Big| \sum_{j \in \calJ} \sum_{J \in \calD^m} \langle b_j, h_J\rangle \{  \langle f_j \rangle_J, \langle g_j, h_J \rangle \}_E \Big| \\
&\le  \sum_{j \in \calJ} \sum_{J \in \calD^m} |\langle b_j, h_J\rangle| | \{  \langle f_j \rangle_J, \langle g_j, h_J \rangle \}_E | \\
&\lesssim \sum_{j \in \calJ} \int_{\R^m} \Big(\sum_{J \in \calD^m} | \{  \langle f_j \rangle_J, \langle g_j, h_J \rangle \}_E |^2 \frac{1_J(y)}{|J|} \Big)^{1/2}\ud y,
\end{align*}
where for each $j$ we used $\|b_j\|_{\BMO_{\calD^m}} \le 1$ via the following inequality
$$
\sum_{J \in \calD^m} |\langle b_j, h_J\rangle|  |a_J| \lesssim \int_{\R^m} \Big(\sum_{J \in \calD^m} |a_J|^2 \frac{1_J(y)}{|J|} \Big)^{1/2}\ud y.
$$
Here $(a_J)_{J \in \calD^m}$ can be an arbitrary sequence of scalars (see Section \ref{ss:bmo}).

For a fixed $y \in \R^m$ we have (by using $\ell^2$ duality) that
\begin{align*}
\Big(\sum_{J \in \calD^m} | \{  \langle f_j \rangle_J, \langle g_j, &h_J \rangle \}_E |^2 \frac{1_J(y)}{|J|} \Big)^{1/2} \\
&\le \Big\{ M_{\calD^m, E} f_j(y), \Big( \sum_{J \in \calD^m} |\langle g_j, h_J \rangle|^2 \frac{1_J(y)}{|J|}  \Big)^{1/2} \Big\}_E.
\end{align*}
Therefore, we have
\begin{align*}
\Big| \sum_{j \in \calJ} \int_{\R^m} &\{\pi_{b_j} f_j(x), g_j(x)\}_E \ud x \Big| \\
&\lesssim \int_{\R^m} \sum_{j \in \calJ} \Big\{ M_{\calD^m, E} f_j(y), \Big( \sum_{J \in \calD^m} |\langle g_j, h_J \rangle|^2 \frac{1_J(y)}{|J|}  \Big)^{1/2} \Big\}_E
\ud y \\
&\le  \int_{\R^m} \Big\{ \Big( \sum_{j \in \calJ} |M_{\calD^m, E} f_j(y)|^2   \Big)^{1/2}, \Big( \sum_{j \in \calJ}
\sum_{J \in \calD^m} |\langle g_j, h_J \rangle|^2 \frac{1_J(y)}{|J|}  \Big)^{1/2} \Big\}_E\ud y \\
&\le \Big\| \Big( \sum_{j \in \calJ} |M_{\calD^m, E} f_j|^2   \Big)^{1/2} \Big\|_{L^p(\R^m; E)} 
 \Big\|  \Big( \sum_{j \in \calJ} \sum_{J \in \calD^m} |\langle g_j, h_J \rangle|^2 \frac{1_J}{|J|}  \Big)^{1/2} \Big\|_{L^{p'}(\R^m; E^*)}.
\end{align*}
That 
\begin{equation*}
\Big\| \Big( \sum_{j \in \calJ} |M_{\calD^m, E} f_j|^2   \Big)^{1/2} \Big\|_{L^p(\R^m; E)}
\lesssim \Big\| \Big( \sum_{j \in \calJ} | f_j|^2   \Big)^{1/2} \Big\|_{L^p(\R^m; E)}
\end{equation*}
holds follows from Proposition \ref{prop:Rubio} (see also the discussion below the proposition).
The next estimate follows from the one-parameter version of Corollary \ref{cor:biparmarSEQ}:
$$
\Big\|  \Big( \sum_{j \in \calJ} \sum_{J \in \calD^m} |\langle g_j, h_J \rangle|^2 \frac{1_J}{|J|}  \Big)^{1/2} \Big\|_{L^{p'}(\R^m; E^*)}
\lesssim \Big\|  \Big( \sum_{j \in \calJ} |g_j|^2  \Big)^{1/2} \Big\|_{L^{p'}(\R^m; E^*)}.
$$
\end{proof}

\subsection{$\calR$-boundedness of bi-parameter paraproducts}
\begin{prop}\label{prop:2parprodstan}
Suppose that $\|b_i\|_{\BMO_{\textup{prod}}^{\calD^m, \calD^k}} \le 1$, $i \in \calI$, $E$ is a UMD function lattice and $p,r \in (1,\infty)$. Then
$$
\calR(\{ \Pi_{\calD^m, \calD^k, b_i} \in \calL(L^p(\R^{m}; L^r(\R^k;E)))\colon i \in \calI\}) \lesssim 1.
$$
\end{prop}
\begin{proof}
Denote $\Pi_{b_i} := \Pi_{\calD^m, \calD^k, b_i}$.
We will show that given a finite $\mathcal{J} \subset \calI$, $f_j \in L^p(\R^{m}; L^r(\R^k;E))$ and $g_j \in L^{p'}(\R^{m}; L^{r'}(\R^k;E^*))$
we have
\begin{align*}
\Big| \sum_{j \in \mathcal{J}}& \iint_{\R^{m+k}} \big\{ \Pi_{b_j} f_j(x_1, x_2),  g_j(x_1,x_2)\big\}_E  \ud x_1 \ud x_2\Big| \\ 
&\lesssim
\Big\| \Big( \sum_{j \in \mathcal{J}} | f_j |^2 \Big)^{1/2} \Big\|_{L^p(\R^{m}; L^r(\R^k;E))}
\Big\| \Big( \sum_{j \in \mathcal{J}} | g_j |^2 \Big)^{1/2} \Big\|_{L^{p'}(\R^{m}; L^{r'}(\R^k;E^*))}.
\end{align*}
That this is enough is justified like in Proposition \ref{prop:1parprodLattice}.

Recall that
$$
\Pi_{b_j} f_j(x_1,x_2) = \mathop{\sum_{V \in \calD^m}}_{U \in \calD^k} \lambda_{V,U}^j \bla f_j \bra_{V \times U} h_V(x_1) h_U(x_2),
$$
where the scalars $\lambda_{V,U}^j$ satisfy for all $j$ and all scalars $A_{V,U}$ that
$$
\mathop{\sum_{V \in \calD^m}}_{U \in \calD^k} |\lambda_{V,U}^j| |A_{V,U}| \lesssim \iint_{\R^{m+k}} \Big( \mathop{\sum_{V \in \calD^m}}_{U \in \calD^k}
|A_{V,U}|^2 \frac{1_V \otimes 1_U}{|V| |U|} \Big)^{1/2}.
$$
This gives
\begin{align*}
\Big| \sum_{j \in \mathcal{J}}& \iint_{\R^{m+k}} \big\{ \Pi_{b_j} f_j,  g_j\big\}_E  \Big|  \\
&\le \sum_{j \in \mathcal{J}} \mathop{\sum_{V \in \calD^m}}_{U \in \calD^k} \big|\lambda_{V,U}^j| |\big\{ \bla f_j \bra_{V \times U}, \bla g_j, h_V \otimes h_U \bra\big\}_E\big| \\
& \lesssim \sum_{j \in \mathcal{J}} \iint_{\R^{m+k}} \Big( \mathop{\sum_{V \in \calD^m}}_{U \in \calD^k}
\big\{ \calM_{\calD^m, \calD^k, E}f_j, \big|\bla g_j, h_V \otimes h_U \bra\big|\big\}_E^2 \frac{1_V\otimes1_U}{|V| |U|} \Big)^{1/2},
\end{align*}
which is further bounded by
\begin{align*}
& \iint_{\R^{m+k}} \sum_{j \in \mathcal{J}} \Big\{ \calM_{\calD^m, \calD^k, E}f_j, \Big( \mathop{\sum_{V \in \calD^m}}_{U \in \calD^k} \big|\bla g_j, h_V \otimes h_U \bra\big|^2 
\frac{1_V\otimes1_U}{|V| |U|} \Big)^{1/2} \Big\}_E \\
&\le \Big\| \Big( \sum_{j \in \mathcal{J}} |\calM_{\calD^m, \calD^k, E}f_j|^2 \Big)^{1/2} \Big\|_{L^p(\R^{m}; L^r(\R^k;E))} \\
&\hspace{2cm}\times \Big\| \Big(\sum_{j \in \mathcal{J}} \mathop{\sum_{V \in \calD^m}}_{U \in \calD^k} \big|\bla g_j, h_V \otimes h_U \bra\big|^2 \frac{1_V\otimes1_U}{|V| |U|} \Big)^{1/2} \Big\|_{L^{p'}(\R^{m}; L^{r'}(\R^k;E^*))}.
\end{align*}
The proof is finished by applying Corollary \ref{cor:FS_Strong} and Corollary \ref{cor:biparmarSEQ}.
\end{proof}

\begin{prop}\label{prop:2parprodmixed}
Suppose $\|b_i\|_{\BMO_{\textup{prod}}^{\calD^m,\calD^k}} \le 1$, $i \in \calI$, $E$ is a UMD function lattice and $p,r  \in (1,\infty)$. Then
$$
\calR(\{ \Pi^{\textup{mixed}}_{\calD^m, \calD^k, b_i} \in \calL(L^p(\R^{m}; L^r(\R^k;E)))\colon i \in \calI\}) \lesssim 1.
$$
\end{prop}
\begin{proof}
Denote $\Pi^{\textup{mixed}}_{b_i} = \Pi^{\textup{mixed}}_{\calD^m, \calD^k, b_i}$.
We will show that given a finite $\mathcal{J} \subset \calI$, $f_j \in L^p(\R^{m}; L^r(\R^k;E))$ and $g_j \in L^{p'}(\R^{m}; L^{r'}(\R^k;E^*))$
we have
\begin{align*}
\Big| \sum_{j \in \mathcal{J}}& \iint_{\R^{m+k}} \big\{ \Pi_{b_j}^{\textup{mixed}} f_j(x_1, x_2),  g_j(x_1,x_2)\big\}_E  \ud x_1 \ud x_2\Big| \\ 
&\lesssim
\Big\| \Big( \sum_{j \in \mathcal{J}} | f_j |^2 \Big)^{1/2} \Big\|_{L^p(\R^{m}; L^r(\R^k;E))}
\Big\| \Big( \sum_{j \in \mathcal{J}} | g_j |^2 \Big)^{1/2} \Big\|_{L^{p'}(\R^{m}; L^{r'}(\R^k;E^*))}.
\end{align*}
That this is enough is justified like in Proposition \ref{prop:1parprodLattice}.

Estimating as in the previous proposition we get
\begin{align*}
&\Big| \sum_{j \in \mathcal{J}} \iint_{\R^{m+k}} \big\{ \Pi_{b_j}^{\textup{mixed}} f_j,  g_j\big\}_E \Big| \\
&\lesssim \sum_{j \in \mathcal{J}} \iint_{\R^{m+k}} \Big( \mathop{\sum_{V \in \calD^m}}_{U \in \calD^k} \Big| \big\{ \bla f_j, h_V \otimes \frac{1_U}{|U|} \bra,
\bla g_j, \frac{1_V}{|V|} \otimes h_U \bra \big\}_E \Big|^2 \frac{1_V \otimes 1_U}{|V| |U|} \Big)^{1/2} \\
&\le \sum_{j \in \mathcal{J}} \iint_{\R^{m+k}} \Big( \mathop{\sum_{V \in \calD^m}}_{U \in \calD^k} \big\{ M_{\calD^k, E} \bla f_j, h_V \bra_1,
 M_{\calD^m, E^*} \bla g_j, h_U \bra_2 \big\}_E^2 \frac{1_V \otimes 1_U}{|V| |U|} \Big)^{1/2}.
\end{align*}
Notice that
\begin{align*}
&\Big( \mathop{\sum_{V \in \calD^m}}_{U \in \calD^k} \big\{ M_{\calD^k, E} \bla f_j, h_V \bra_1,
 M_{\calD^m, E^*} \bla g_j, h_U \bra_2 \big\}_E^2 \frac{1_V \otimes 1_U}{|V| |U|} \Big)^{1/2} \\
 &\le \Big( \sum_{V \in \calD^m} \Big\{ \frac{1_V}{|V|^{1/2}} \otimes M_{\calD^k, E} \bla f_j, h_V \bra_1, 
 \Big( \sum_{U \in \calD^k} \big[ M_{\calD^m, E^*} \bla g_j, h_U \bra_2 \big]^2 \otimes \frac{1_U}{|U|} \Big)^{1/2} \Big\}_E^2 \Big)^{1/2} \\
 &\le \Big\{ \Big( \sum_{V \in \calD^m} \frac{1_V}{|V|} \otimes \big[  M_{\calD^k, E} \bla f_j, h_V \bra_1 \big]^2 \Big)^{1/2}, 
  \Big( \sum_{U \in \calD^k} \big[ M_{\calD^m, E^*} \bla g_j, h_U \bra_2 \big]^2 \otimes \frac{1_U}{|U|} \Big)^{1/2} \Big\}_E,
\end{align*}
so that
\begin{align*}
\Big| \sum_{j \in \mathcal{J}}& \iint_{\R^{m+k}} \big\{ \Pi_{b_j}^{\textup{mixed}} f_j(x_1, x_2),  g_j(x_1,x_2)\big\}_E  \ud x_1 \ud x_2\Big| \\
&\lesssim \Big\|  \Big(  \sum_{j \in \mathcal{J}} \sum_{V \in \calD^m} \frac{1_V}{|V|} \otimes \big[  M_{\calD^k, E} \bla f_j, h_V \bra_1 \big]^2 \Big)^{1/2} \Big\|_{L^p(\R^{m}; L^r(\R^k;E))} \\
&\hspace{2cm}\times \Big\| \Big( \sum_{j \in \mathcal{J}}\sum_{U \in \calD^k} \big[ M_{\calD^m, E^*} \bla g_j, h_U \bra_2 \big]^2 \otimes \frac{1_U}{|U|}  \Big)^{1/2} \Big\|_{L^{p'}(\R^{m}; L^{r'}(\R^k;E^*))}.
\end{align*}

There holds
\begin{equation*}
\begin{split}
\Big\|  \Big(  \sum_{j \in \mathcal{J}} \sum_{V \in \calD^m} & \frac{1_V}{|V|} \otimes \big[  M_{\calD^k, E} \bla f_j, h_V \bra_1 \big]^2 \Big)^{1/2} \Big\|_{L^p(\R^{m}; L^r(\R^k;E))} \\
& \lesssim \Big\|  \Big(  \sum_{j \in \mathcal{J}} \sum_{V \in \calD^m} \frac{1_V}{|V|} \otimes \big[  \bla f_j, h_V \bra_1 \big]^2 \Big)^{1/2} \Big\|_{L^p(\R^{m}; L^r(\R^k;E))} \\
& \lesssim \Big\|  \Big(  \sum_{j \in \mathcal{J}}  | f_j|^2 \Big)^{1/2} \Big\|_{L^p(\R^{m}; L^r(\R^k;E))}.
\end{split}
\end{equation*}
In the first step we used Proposition \ref{prop:Rubio} (see again also the discussion after the proposition) in the inner integral over $\R^k$. 
The second step was an application of Corollary \ref{cor:biparmarSEQ}. 
The term related to the sequence $\{g_j\}$ is handled with a corresponding argument, except that the first step is now an application of Corollary \ref{cor:FS_M1}.
\end{proof}
\begin{rem}
The estimate in $L^r(\R^{k}; L^p(\R^m;E))$ follows with the same proof.
\end{rem}

\section{Bi-parameter partial paraproducts}
A bi-parameter partial paraproduct  is an operator of the form
$$
P_{\calD^n, \calD^m}^{i_1, i_2} f = \sum_{K \in \calD^n} \mathop{\sum_{I_1, I_2 \in \calD^n}}_{I_1^{(i_1)} = I_2^{(i_2)} = K} h_{I_2} \otimes
\pi_{\calD^m, b_{K, I_1,I_2}}(\langle f, h_{I_1} \rangle_1), \qquad f \in L^1_{\loc}(\R^{n+m};E),
$$
where $E$ is a UMD space with the property $(\alpha)$ of Pisier and
$\pi_{\calD^m, b_{K, I_1,I_2}}$ is a dyadic (one-parameter) paraproduct for some function
$$
b_{K, I_1,I_2} \colon \R^m \to \R
$$
satisfying
$$
\|b_{K, I_1,I_2}\|_{\BMO_{\calD^m}} \le \frac{|I_1|^{1/2} |I_2|^{1/2}}{|K|}.
$$

With $E = \R$ (or $\C$) this is the exact form in which these operators appear in the bi-parameter representation theorem \cite{Ma1}.
Of course, such operators appear also in the form that contains the dual paraproducts $\pi_{\calD^m, b_{K, I_1,I_2}}^*$, and in the form that the paraproduct component is in $\R^n$.
\begin{thm}\label{thm:biparPartialParaProd}
Let $E$ be a UMD function lattice, $p,q \in (1,\infty)$ and $i_1, i_2 \ge 0$. 
Suppose that for each $k \in \calK$ we are given a partial paraproduct
$$
P_{\calD^n, \calD^m, k}^{i_1, i_2} f = \sum_{K \in \calD^n} \mathop{\sum_{I_1, I_2 \in \calD^n}}_{I_1^{(i_1)} = I_2^{(i_2)} = K} h_{I_2} \otimes
\pi_{\calD^m, b_{K, I_1,I_2, k}}(\langle f, h_{I_1} \rangle_1), \qquad f \in L^1_{\loc}(\R^{n+m};E),
$$
where
$$
\|b_{K, I_1,I_2, k}\|_{\BMO_{\calD^m}} \le \frac{|I_1|^{1/2} |I_2|^{1/2}}{|K|}.
$$
Then we have
$$
\calR(\{ P_{\calD^n, \calD^m, k}^{i_1, i_2} \in \calL(L^q(\R^n; L^p(\R^m;E))) \colon k \in \calK\}) \lesssim \min(i_1, i_2) + 1
$$
and
$$
\calR(\{ P_{\calD^n, \calD^m, k}^{i_1, i_2} \in \calL(L^p(\R^m; L^q(\R^n; E))) \colon k \in \calK\}) \lesssim \min(i_1, i_2) + 1.
$$
\end{thm}
\begin{proof}
Fix $p \in (1,\infty)$. We can see the partial paraproduct $P_{\calD^n, \calD^m,k}^{i_1, i_2}$, $k \in \calK$, as a model operator when it acts
on locally integrable functions $f \colon \R^n \to F$, where $F = L^p(\R^m; E)$. Proposition \ref{prop:mod1-1} says that for all $q \in (1,\infty)$ we have
$$
\calR(\{ P_{\calD^n, \calD^m, k}^{i_1, i_2} \in \calL(L^q(\R^n; L^p(\R^m;E))) \colon k \in \calK\}) \lesssim (\min(i_1, i_2) + 1) R_p(E),
$$
where
$$
R_p(E) := \calR\Big( \Big\{\frac{|K|}{|I_1|^{1/2} |I_2|^{1/2}} \pi_{\calD^m, b_{K, I_1,I_2, k}}  \in \calL(L^p(\R^m; E))\colon k \in \calK, K, I_1, I_2 \in \calD^n \Big\} \Big).
$$
Proposition \ref{prop:1parprodLattice} says that $R_p(E) \lesssim 1$, as $E$ is a UMD function lattice.

On the other hand, Proposition \ref{prop:mod1-2} says that for all $p,q \in (1,\infty)$ we have
$$
\calR(\{ P_{\calD^n, \calD^m, k}^{i_1, i_2} \in \calL(L^p(\R^m; L^q(\R^n; E))) \colon k \in \calK\}) \lesssim (\min(i_1, i_2) + 1) R_{p,q}(E),
$$
where
\begin{align*}
R_{p,q}(E) = 
\calR\Big(\Big\{ \frac{|K|}{|I_1|^{1/2} |I_2|^{1/2}} \pi_{\calD^m, b_{K, I_1,I_2, k}} \in \calL(L^p(\R^m&; L^q(\R^n \times Y^n; E)))\colon \\
& k \in \calK, K, I_1, I_2 \in \calD^n\Big\} \Big).
\end{align*}
Since $E$ is a UMD function lattice, we can apply Proposition \ref{prop:1parprodLattice} with $L^q(\R^n \times Y^n; E)$. This gives that $R_{p,q}(E) \lesssim 1$.
\end{proof}
\begin{rem}
We wrote the proof like we did to illustrate the following point.
Notice that even if we would be interested only in the case $E = \R$, the second part of the proof
would require applying Proposition \ref{prop:1parprodLattice} to the UMD function lattice $L^q(\R^n \times Y^n)$.
That is, we would need the $\calR$-boundedness of ordinary paraproducts in some vector-valued setting anyway! Of course,
in the $p =q$ case we could just use the easier proof i.e. the first part of the proof.
\end{rem}

\section{Application to bi-parameter singular integrals}
The definition of a bi-parameter singular integral $T$ is somewhat lengthy. We only give a brief idea here, for the full details see \cite{Ma1}.

The definition involves first of all the structural
assumption that $T$ should have a full kernel representation i.e. $\langle Tf, g \rangle$ can be written as
as integral operator with a kernel $K\colon (\R^{n+m} \times \R^{n+m}) \setminus \{(x,y) \in \R^{n+m}  \times \R^{n+m}:\, x_1 = y_1 \textup{ or } x_2 = y_2\} \to \R$, when $f = f_1 \otimes f_2$ and $g = g_1 \otimes g_2$ are of tensor product form with $\operatorname{spt}\,f_1\cap \operatorname{spt}\,g_1 = \emptyset$ and $ \operatorname{spt}\,f_2 \cap  \operatorname{spt}\,g_2 = \emptyset$. The kernel $K$ needs to satisfy various estimates: the size estimate, H\"older estimates, and 
mixed H\"older and size estimates. Some Calder\'on--Zygmund structure on $\R^n$ and $\R^m$ is also demanded separately. This entails
having regular enough kernel representations when only  $\operatorname{spt}\,f_1\cap \operatorname{spt}\,g_1 = \emptyset$ or $\operatorname{spt}\,f_2 \cap  \operatorname{spt}\,g_2 = \emptyset$.

Next, we demand various boundedness and cancellation assumptions: the weak boundedness assumption, some diagonal BMO conditions, and finally
that $T1, T^*1, T_1(1)$ and $T_1^*(1)$ belong to $\BMO_{\textup{prod}}(\R^{n+m})$. Here $T_1$ is the first partial adjoint of $T$ i.e.
$\langle T_1(f_1 \otimes f_2), g_1 \otimes g_2 \rangle = \langle T(g_1 \otimes f_2), f_1 \otimes g_2\rangle$.

With such assumptions it was shown in \cite{Ma1} that a specific dyadic representation of bi-parameter singular integrals holds.
For all bounded and compactly supported $f,g \colon \R^{n+m} \to \R$, we have
$$
\langle Tf, g\rangle = C_T \mathbb{E}_{w_n} \mathbb{E}_{w_m}\mathop{\sum_{(i_1, i_2) \in \Z_+^2}}_{(j_1, j_2) \in \Z_+^2} \alpha_{i_1,i_2,j_1,j_2} 
\sum_{u}
\langle S^{i_1,i_2,j_1,j_2}_{\mathcal{D}^n,\mathcal{D}^m, u}f, g \rangle,
$$
where $\alpha_{i_1, i_2, j_1, j_2} = 2^{-\max(i_1,i_2)\delta/2}2^{-\max(j_1,j_2)\delta/2}$ (the $\delta$ appears on the H\"older estimates of the kernels)
and $u$ runs over some finitely many integers. Here
$S^{i_1,i_2,j_1,j_2}_{\calD^n,\calD^m, u}$ is a bi-parameter shift  if $(i_1, i_2) \ne (0,0)$ and $(j_1, j_2) \ne (0,0)$.
More specifically, we mean a cancellative bi-parameter shift (such as we have defined in this paper even in the operator-valued setting) with the scalar-valued
kernels being uniformly bounded by $1$.
In the other remaining cases $S^{i_1,i_2,j_1,j_2}_{\calD^n,\calD^m, u}$ can either be
a bi-parameter shift, some partial paraproduct
with the paraproduct component in $\R^n$ or $\R^m$ (BMO bounds normalised like $\le (|I_1|^{1/2} |I_2|^{1/2}) / |K|$ 
or $(|J_1|^{1/2} |J_2|^{1/2})/|V|$)),
or a full bi-parameter paraproduct, either of standard type or of mixed type,
associated with a product BMO function of norm at most 1. We have one full standard paraproduct, one adjoint of such, one mixed paraproduct and an adjoint of such, and
these appear in the case $(i_1, i_2, j_1, j_2) = (0,0,0,0)$.
Moreover, the average is taken over all random dyadic grids $\calD^n = \calD^n(w_n)$
and $\calD^m = \calD^m(w_m)$. 

If $E$ is a UMD space with the property $(\alpha)$ of Pisier, we can apply this to
simple functions $f = \sum_{a=1}^A f_a e_a$ and $g = \sum_{b=1}^B g_b e^*_b$, where
$f_a \colon \R^{n+m} \to \R$ and $g_b \colon \R^{n+m} \to \R$ are bounded and compactly supported,
$e_a \in E$ and $e^*_b \in E^*$. This gives that
\begin{align*}
&\iint_{\R^{n+m}} \{ Tf(x), g(x)\}_E\ud x \\
&=  C_T \mathbb{E}_{w_n} \mathbb{E}_{w_m}\mathop{\sum_{(i_1, i_2) \in \Z_+^2}}_{(j_1, j_2) \in \Z_+^2} \alpha_{i_1,i_2,j_1,j_2} \sum_u
\iint_{\R^{n+m}} \{ S^{i_1,i_2,j_1,j_2}_{\mathcal{D}^n,\mathcal{D}^m, u}f(x), g(x)\}_E\ud x.
\end{align*}
Here the interpretation is clear: $Tf(x_1,x_2) = \sum_{a=1}^A Tf_a(x_1,x_2) e_a$ (and the same with the shifts).
\begin{thm}\label{thm:boundednessBiParSIORbound}
Suppose $\calI$ is a index set, and that for each $i \in \calI$ we are given a bi-parameter SIO as in \cite{Ma1}. Suppose that the H\"older exponents
of the kernels of these operators are uniformly bounded from below, and that all the other constants in the assumptions are uniformly bounded from above.
Let $E$ be a UMD function lattice, and $p,q \in (1,\infty)$.
Then we have
$$
\calR(\{ T_i \in \calL(L^q(\R^n; L^p(\R^m;E))) \colon i \in \calI\}) \lesssim 1.
$$
\end{thm}
\begin{proof}
Fix a finite $\calK \subset \calI$ and simple functions $f_k \colon \R^{n+m} \to E$, $k \in \calK$. We need to prove that
$$
\Big(\E \Big\| \sum_{k \in \calK} \epsilon_k T_k f_k \Big\|^2_{L^q(\R^n; L^p(\R^m;E))}\Big)^{1/2} \lesssim
 \Big(\E \Big\| \sum_{k \in \calK} \epsilon_k f_k \Big\|^2_{L^q(\R^n; L^p(\R^m;E))}\Big)^{1/2}.
$$
Lemma \ref{lem:randomduality} says that the right hand side is dominated by
$$
\sup \Big| \sum_{k \in \calK} \iint_{\R^{n+m}} \{ T_kf_k(x), g_k(x)\}_E\ud x \Big|,
$$
where the supremum is taken over those simple $g_k \colon \R^{n+m} \to E^*$ for which
$$
\Big(\E \Big\| \sum_{k \in \calK} \epsilon_k g_k \Big\|^2_{L^{q'}(\R^n; L^{p'}(\R^m;E^*))}\Big)^{1/2} \le 1.
$$
With any fixed such sequence $(g_k)_{k \in \calK}$ we write
\begin{align*}
&\Big| \sum_{k \in \calK} \iint_{\R^{n+m}} \{ T_kf_k(x), g_k(x)\}_E\ud x \Big| \\
&= \Big| \sum_{k \in \calK} C \mathbb{E}_{w_n} \mathbb{E}_{w_m}\mathop{\sum_{(i_1, i_2) \in \Z_+^2}}_{(j_1, j_2) \in \Z_+^2} \alpha_{i_1,i_2,j_1,j_2} \sum_u
\iint_{\R^{n+m}} \{ S^{i_1,i_2,j_1,j_2}_{\calD^n,\calD^m, k, u}f_k(x), g_k(x)\}_E\ud x \Big| \\
&\le C  \mathbb{E}_{w_n} \mathbb{E}_{w_m}\mathop{\sum_{(i_1, i_2) \in \Z_+^2}}_{(j_1, j_2) \in \Z_+^2} \alpha_{i_1,i_2,j_1,j_2} \sum_u
\Big| \sum_{k \in \calK} \iint_{\R^{n+m}} \{ S^{i_1,i_2,j_1,j_2}_{\calD^n,\calD^m, k, u}f_k(x), g_k(x)\}_E\ud x \Big|.
\end{align*}
We have
\begin{align*}
\Big| &\sum_{k \in \calK} \iint_{\R^{n+m}} \{ S^{i_1,i_2,j_1,j_2}_{\calD^n,\calD^m, k, u}f_k(x), g_k(x)\}_E\ud x \Big|\\
&= \Big| \E  \iint_{\R^{n+m}} \Big\{ \sum_{k \in \calK} \epsilon_k S^{i_1,i_2,j_1,j_2}_{\calD^n,\calD^m, k, u}f_k(x), \sum_{k' \in \calK} \epsilon_{k'}
g_{k'}(x)\Big\}_E\ud x \Big| \\
&\le \Big(\E \Big\| \sum_{k \in \calK} \epsilon_k S^{i_1,i_2,j_1,j_2}_{\calD^n,\calD^m, k, u} f_k \Big\|^2_{L^q(\R^n; L^p(\R^m;E))}\Big)^{1/2}.
\end{align*}
If for the fixed $i_1, i_2, j_1, j_2$ and $u$
the appearing operators $S^{i_1,i_2,j_1,j_2}_{\calD^n,\calD^m, k, u}$, $k \in \calK$, are bi-parameter shifts,
Proposition \ref{prop:biparOP} (a special case where the kernels are scalar-valued and pointwise uniformly bounded by $1$) gives
\begin{align*}
\Big(\E \Big\| \sum_{k \in \calK} \epsilon_k& S^{i_1,i_2,j_1,j_2}_{\calD^n,\calD^m, k, u} f_k \Big\|^2_{L^q(\R^n; L^p(\R^m;E))}\Big)^{1/2} \\
&\lesssim (\min(i_1, i_2)+1)(\min(j_1, j_2)+1) \Big(\E \Big\| \sum_{k \in \calK} \epsilon_k f_k \Big\|^2_{L^q(\R^n; L^p(\R^m;E))}\Big)^{1/2}.
\end{align*}
If they are partial paraproducts, then Theorem \ref{thm:biparPartialParaProd} gives the same bound. Notice that even if we have here arbitrarily chosen the symmetry $L^q(\R^n; L^p(\R^m;E))$ (and not $L^p(\R^m; L^q(\R^n;E))$), we always also need the second inequality of Theorem \ref{thm:biparPartialParaProd} here (to handle the case $S^{0, 0, j_1,j_2}_{\calD^n, \calD^m,k, u}$ where the paraproduct component is $\R^n$). Finally, Propositions \ref{prop:2parprodstan} and \ref{prop:2parprodmixed} give the same bound in the case that the operators $S^{0, 0, 0,0}_{\calD^n, \calD^m, k, u}$, $k \in \calK$, are full paraproducts of the same type. The proof is complete.
\end{proof}

\section{Tri-parameter partial paraproducts}
We conclude the paper by studying boundedness properties of tri-parameter partial paraproducts. They come in essentially two different forms, here called type 1 and type 2.
The proved estimates are a key step in proving that $T1$ type assumptions for tri-parameter singular integrals
directly imply that tri-parameter singular integrals are $L^q(\R^n; L^p(\R^m; L^r(\R^k;E)))$ bounded for all $p, q, r \in (1,\infty)$ and UMD function lattices $E$, that is, for
proving the analog of Theorem \ref{thm:boundednessBiParSIORbound} for tri-parameter singular integrals.
In addition to these partial paraproducts one would also need to consider the full tri-parameter paraproducts of different flavours and tri-parameter shifts. The shifts
would be straightforward to handle, and we believe that the full paraproducts should not pose problems either.
\subsection{Tri-parameter partial paraproducts of type 1}
A tri-parameter partial paraproduct of type 1 is an operator of the form
$$
P^{i_1, i_2}_{\calD^n, \calD^m, \calD^k} f = \sum_{K \in \calD^n} \mathop{\sum_{I_1, I_2 \in \calD^n}}_{I_1^{(i_1)} = I_2^{(i_2)} = K} h_{I_2} \otimes B_{\calD^m, \calD^k, K, I_1, I_2}(\langle f, h_{I_1} \rangle_1),
$$
where $f \in L^1_{\loc}(\R^{n+m+k};E)$, $E$ is a UMD space satisfying Pisier's property $(\alpha)$,
and $B_{\calD^m, \calD^k,K, I_1, I_2}$ is
a standard full bi-parameter paraproduct $\Pi_{\calD^m, \calD^k, b_{K, I_1,I_2}}$  for all $K, I_1, I_2$,
or a mixed full bi-parameter paraproduct $\Pi^{\textup{mixed}}_{\calD^m, \calD^k, b_{K, I_1,I_2}}$ for all $K, I_1, I_2$.
The functions
$$
b_{K, I_1,I_2} \colon \R^{m+k} \to \R
$$
satisfy
$$
\|b_{K, I_1,I_2}\|_{\BMO_{\textup{prod}}^{\calD^m, \calD^k}} \le \frac{|I_1|^{1/2} |I_2|^{1/2}}{|K|}.
$$
\begin{prop}
Let $E$ be a UMD function lattice, $p,q, r \in (1,\infty)$ and $i_1, i_2 \ge 0$. 
Suppose that for each $s \in \calS$ we are given a tri-parameter partial paraproduct of type 1
$$
P^{i_1, i_2}_{\calD^n, \calD^m, \calD^k, s} f = \sum_{K \in \calD^n} \mathop{\sum_{I_1, I_2 \in \calD^n}}_{I_1^{(i_1)} = I_2^{(i_2)} = K} h_{I_2} \otimes B_{\calD^m, \calD^k, K, I_1, I_2, s}(\langle f, h_{I_1} \rangle_1),
$$
where either $B_{\calD^m, \calD^k, K, I_1, I_2, s} = \Pi_{\calD^m, \calD^k, b_{K, I_1,I_2, s}}$ for all $s \in \calS$ and $K, I_1, I_2 \in \calD^n$
or $B_{\calD^m, \calD^k, K, I_1, I_2, s} = \Pi^{\textup{mixed}}_{\calD^m, \calD^k, b_{K, I_1,I_2,s}}$ for all $s \in \calS$ and $K, I_1, I_2 \in \calD^n$.
Moreover, assume that
$$
\|b_{K, I_1,I_2,s}\|_{\BMO_{\textup{prod}}^{\calD^m, \calD^k}} \le \frac{|I_1|^{1/2} |I_2|^{1/2}}{|K|}.
$$
Then we have
$$
\calR(\{ P^{i_1, i_2}_{\calD^n, \calD^m, \calD^k, s} \in \calL(L^q(\R^n; L^p(\R^{m}; L^r(\R^k;E)))) \colon s \in \calS\}) \lesssim \min(i_1, i_2) + 1,
$$
and the same is true with all the other permutations of $L^q$, $L^p$ and $L^r$.
\end{prop}
\begin{proof}
We view $P^{i_1, i_2}_{\calD^n, \calD^m, \calD^k, s}$
as a model operator acting
on locally integrable functions $f \colon \R^n \to F$, where $F = L^p(\R^{m}; L^r(\R^k; E))$.
Proposition \ref{prop:mod1-1} says that
\begin{align*}
\calR(\{ P^{i_1, i_2}_{\calD^n, \calD^m, \calD^k, s} \in \calL(L^q(\R^n; L^p(\R^{m}; &L^r(\R^k;E)))) \colon s \in \calS\}) \\
&\lesssim (\min(i_1, i_2) + 1) R_{p,r}(E),
\end{align*}
where $R_{p,r}(E)$ is the constant
\begin{align*}
 \calR\Big( \Big\{\frac{|K|}{|I_1|^{1/2} |I_2|^{1/2}} B_{\calD^m, \calD^k,K, I_1, I_2,s} 
 \in \calL(L^p(\R^{m}; L^r(\R^k; E)))\colon s \in \calS, K, I_1, I_2 \in \calD^n \Big\} \Big).
\end{align*}
Propositions \ref{prop:2parprodstan} and \ref{prop:2parprodmixed} say that $R_{p,r}(E) \lesssim 1$.

Proposition \ref{prop:mod1-2} says that
\begin{align*}
\calR(\{ P^{i_1, i_2}_{\calD^n, \calD^m, \calD^k, s} \in \calL(L^p(\R^m; L^q(\R^{n}; &L^r(\R^k;E)))) \colon s \in \calS\}) \\
&\lesssim (\min(i_1, i_2) + 1) R_{p,q,r}^1(E),
\end{align*}
where $R_{p,q,r}^1(E)$ is the constant
\begin{align*}
 \calR\Big( \Big\{\frac{|K|}{|I_1|^{1/2} |I_2|^{1/2}} B_{\calD^m, \calD^k,K, I_1, I_2,s} 
 \in \calL(L^p(\R^{m}; L^q(\R^n \times& Y^n; L^r(\R^k; E))))\\
 &\colon s \in \calS, K, I_1, I_2 \in \calD^n \Big\} \Big).
\end{align*}
To bound this constant we need a bit strange versions of Propositions \ref{prop:2parprodstan} and \ref{prop:2parprodmixed}.
So we need to check the $\calR$-boundedness of the extensions of these full bi-parameter paraproducts in a space of the form $L^p(\R^{m}; L^q(X; L^r(\R^k; E)))$, where
$X$ is some measure space. Let $f_j = \sum_{a \in \mathcal{A}_j} 1_{X_{j,a}} F_{j,a}$, where $F_{j,a} \colon \R^{m+k} \to E$ and $X_{j,a} \subset X$, and
$g_j =  \sum_{c \in \mathcal{C}_j} 1_{\tilde X_{j,c}} G_{j,c}$, where $G_{j,c }  \colon \R^{m+k} \to E^*$ and $\tilde X_{j,c} \subset X$.
Following the proofs of the said propositions one ends up with the need to show that
$$
\Big\| \Big( \sum_j \big[  \calM_{\calD^m, \calD^k, E}^{1,3} f_j \big]^2 \Big)^{1/2} \Big\|_{L^p(\R^{m}; L^q(X; L^r(\R^k; E)))}
$$
and
$$
\Big\| \Big( \sum_j 
 \mathop{\sum_{V \in \calD^m}}_{U \in \calD^k} \big|\bla g_j, h_V \otimes h_U \bra_{1,3}\big|^2 \frac{1_V\otimes1_U}{|V| |U|} \Big)^{1/2}
  \Big\|_{L^{p'}(\R^{m}; L^{q'}(X; L^{r'}(\R^k; E^*)))}
$$
are bounded by
$$
\Big\| \Big( \sum_j |f_j|^2 \Big)^{1/2} \Big\|_{L^p(\R^{m}; L^q(X; L^r(\R^k; E)))} \,\,\,\, \textup{and} \,\,\,\, \Big\| \Big( \sum_j |g_j|^2 \Big)^{1/2} \Big\|_{L^{p'}(\R^{m}; L^{q'}(X; L^{r'}(\R^k; E^*)))}
$$
respectively, and the same with
$$
\Big\| \Big( \sum_j  
 \sum_{V \in \calD^m}  \Big[  M_{\calD^k, E}^3 \Big( \frac{1_V}{|V|^{1/2}} \otimes \bla f_{j}, h_V \bra_1\Big)\Big]^2 \Big)^{1/2} \Big\|_{L^p(\R^{m}; L^q(X; L^r(\R^k; E)))}
$$
 and
$$
\Big\| \Big( \sum_j  
 \sum_{U \in \calD^k}  \Big[  M_{\calD^m, E^*}^1 \Big(  \bla g_{j}, h_U \bra_3 \otimes \frac{1_U}{|U|^{1/2}} \Big)\Big]^2 \Big)^{1/2} \Big\|_{L^{p'}(\R^{m}; L^{q'}(X; L^{r'}(\R^k; E^*)))}.
$$
But these are all bounded essentially in the same way as in the original propositions. Therefore, we have $R_{p,q,r}^1(E) \lesssim 1$.

Proposition \ref{prop:mod1-3} says that
\begin{align*}
\calR(\{ P^{i_1, i_2}_{\calD^n, \calD^m, \calD^k, s} \in \calL(L^p(\R^m; L^r(\R^{k}; &L^q(\R^n;E)))) \colon s \in \calS\}) \\
&\lesssim (\min(i_1, i_2) + 1) R_{p,q,r}^2(E),
\end{align*}
where $R_{p,q,r}^2(E)$ is the constant
\begin{align*}
 \calR\Big( \Big\{\frac{|K|}{|I_1|^{1/2} |I_2|^{1/2}} B_{\calD^m, \calD^k,K, I_1, I_2,s} 
 \in \calL(L^p(\R^{m}; L^r(\R^k&; L^p(\R^n \times Y^n; E))))\\
 &\colon s \in \calS, K, I_1, I_2 \in \calD^n \Big\} \Big).
\end{align*}
That $R_{p,q,r}^2(E) \lesssim 1$ follows from Propositions \ref{prop:2parprodstan} and \ref{prop:2parprodmixed} applied
with $E$ replaced by $L^q(\R^n \times Y^n;E)$.

We can do all of the above but just starting with $F = L^r(\R^{k}; L^p(\R^m; E))$, so all of the symmetries follow.
\end{proof}

\subsection{Tri-parameter partial paraproducts of type 2}
A tri-parameter partial paraproduct of type 2 is an operator of the form
\begin{align*}
&P^{i_1, i_2, j_1, j_2}_{\calD^n, \calD^m, \calD^k} f \\ &= \sum_{K \in \calD^n} \sum_{V \in \calD^m}
 \mathop{\sum_{I_1, I_2 \in \calD^n}}_{I_1^{(i_1)} = I_2^{(i_2)} = K}
 \mathop{\sum_{J_1, J_2 \in \calD^m}}_{J_1^{(j_1)} = J_2^{(j_2)} = V}
  h_{I_2} \otimes h_{J_2} \otimes  \pi_{\calD^k, b_{K, V, I_1,I_2, J_1, J_2}}(\langle f, h_{I_1} \otimes h_{J_1} \rangle_{1,2}),
\end{align*}
where $f \in L^1_{\loc}(\R^{n+m+k};E)$, $E$ is a UMD space satisfying Pisier's property $(\alpha)$,
and $\pi_{\calD^k, b_{K, V, I_1,I_2, J_1, J_2}}$ is a dyadic  paraproduct 
for some function
$$
b_{K, V, I_1,I_2, J_1, J_2} \colon \R^k \to \R
$$
satisfying
$$
\|b_{K, V, I_1,I_2, J_1, J_2}\|_{\BMO_{\calD^k}} \le \frac{|I_1|^{1/2} |I_2|^{1/2}}{|K|} \frac{|J_1|^{1/2} |J_2|^{1/2}}{|V|}.
$$
\begin{prop}
Let $E$ be a UMD function lattice, $p,q, r \in (1,\infty)$ and $i_1, i_2, j_1, j_2 \ge 0$.
Suppose that for each $s \in \calS$ we are given a tri-parameter partial paraproduct of type 2
\begin{align*}
&P^{i_1, i_2, j_1, j_2}_{\calD^n, \calD^m, \calD^k, s} f \\ &= \sum_{K \in \calD^n} \sum_{V \in \calD^m}
 \mathop{\sum_{I_1, I_2 \in \calD^n}}_{I_1^{(i_1)} = I_2^{(i_2)} = K}
 \mathop{\sum_{J_1, J_2 \in \calD^m}}_{J_1^{(j_1)} = J_2^{(j_2)} = V}
  h_{I_2} \otimes h_{J_2} \otimes  \pi_{\calD^k, b_{K, V, I_1,I_2, J_1, J_2, s}}(\langle f, h_{I_1} \otimes h_{J_1} \rangle_{1,2}),
\end{align*}
where
$$
\|b_{K, V, I_1,I_2, J_1, J_2, s}\|_{\BMO_{\calD^k}} \le \frac{|I_1|^{1/2} |I_2|^{1/2}}{|K|} \frac{|J_1|^{1/2} |J_2|^{1/2}}{|V|}.
$$
Then we have
\begin{align*}
\calR(\{ P^{i_1, i_2, j_1, j_2}_{\calD^n, \calD^m, \calD^k, s}\in \calL(L^q(\R^n;& L^p(\R^{m}; L^r(\R^k;E)))) \colon s \in \calS\})  \\
&\lesssim (\min(i_1, i_2) + 1)(\min(j_1, j_2) + 1),
\end{align*}
and the same is true with all the other permutations of $L^q$, $L^p$ and $L^r$.
\end{prop}
\begin{proof}
Define for each $K, I_1, I_2 \in \calD^n$ and $s \in \calS$ the bi-parameter partial paraproduct
$$
P^{j_1, j_2}_{\calD^m, \calD^k, K, I_1, I_2, s} g := \sum_{V \in \calD^m}  \mathop{\sum_{J_1, J_2 \in \calD^m}}_{J_1^{(j_1)} = J_2^{(j_2)} = V}
h_{J_2} \otimes \pi_{\calD^k, b_{K, V, I_1,I_2, J_1, J_2, s}} ( \langle g, h_{J_1} \rangle_1 ),
$$
where $g \in L^1_{\loc}(\R^{m+k}; E$). We have by Theorem \ref{thm:biparPartialParaProd} that
\begin{align*}
&\calR\Big(\Big\{ \frac{|K|}{|I_1|^{1/2}|I_2|^{1/2}} P^{j_1, j_2}_{\calD^m, \calD^k, K, I_1, I_2, s} 
\in \calL(L^p(\R^m; L^r(\R^k;E))) \colon s \in \calS, K, I_1, I_2 \in \calD^n\Big\}\Big) \\
&\lesssim \min(j_1, j_2) + 1
\end{align*}
and
\begin{align*}
&\calR\Big(\Big\{ \frac{|K|}{|I_1|^{1/2}|I_2|^{1/2}} P^{j_1, j_2}_{\calD^m, \calD^k, K, I_1, I_2, s} 
\in \calL(L^r(\R^k; L^p(\R^m;E))) \colon s \in \calS, K, I_1, I_2 \in \calD^n\Big\}\Big) \\
&\lesssim \min(j_1, j_2) + 1.
\end{align*}
Our original operators
$$
P^{i_1, i_2, j_1, j_2}_{\calD^n, \calD^m, \calD^k, s} f = \sum_{K \in \calD^n} \mathop{\sum_{I_1, I_2 \in \calD^n}}_{I_1^{(i_1)} = I_2^{(i_2)} = K}
h_{I_2} \otimes P^{j_1, j_2}_{\calD^m, \calD^k, K, I_1, I_2, s}(\langle f, h_{I_1} \rangle_1),
$$
where $f \in L^1_{\loc}(\R^{n+m+k}; E)$, are model operators of the above form. It follows from Proposition \ref{prop:mod1-1} that
\begin{align*}
\calR(\{ P^{i_1, i_2, j_1, j_2}_{\calD^n, \calD^m, \calD^k, s}\in \calL(L^q(\R^n;& L^p(\R^{m}; L^r(\R^k;E)))) \colon s \in \calS\})  \\
&\lesssim (\min(i_1, i_2) + 1)(\min(j_1, j_2) + 1)
\end{align*}
and
\begin{align*}
\calR(\{ P^{i_1, i_2, j_1, j_2}_{\calD^n, \calD^m, \calD^k, s}\in \calL(L^q(\R^n;& L^r(\R^{k}; L^p(\R^m;E)))) \colon s \in \calS\})  \\
&\lesssim (\min(i_1, i_2) + 1)(\min(j_1, j_2) + 1).
\end{align*}
Proposition \ref{prop:mod1-2} gives
\begin{align*}
&\calR(\{ P^{i_1, i_2, j_1, j_2}_{\calD^n, \calD^m, \calD^k, s}\in \calL(L^r(\R^k; L^q(\R^{n}; L^p(\R^m;E)))) \colon s \in \calS\})  \\
&\lesssim (\min(i_1, i_2) + 1)\calR\Big(\Big\{ \frac{|K|}{|I_1|^{1/2}|I_2|^{1/2}} P^{j_1, j_2}_{\calD^m, \calD^k, K, I_1, I_2, s} \\
&  \hspace{3cm} \in \calL(L^r(\R^k; L^q(\R^n \times Y^n; L^p(\R^m;E))))  \colon s \in \calS, K, I_1, I_2 \in \calD^n\Big\}\Big).
\end{align*}
Viewing $\pi_{\calD^k, b_{K, V, I_1,I_2, J_1, J_2, s}}$ as a bounded operator in $F = L^r(\R^k; L^q(\R^n \times Y^n; E))$ and
$P^{j_1, j_2}_{\calD^m, \calD^k, K, I_1, I_2, s}$ as a model operator in $\R^m$
composed of the paraproduct operators $\pi_{\calD^k, b_{K, V, I_1,I_2, J_1, J_2, s}}$,
we see using Proposition \ref{prop:mod1-3} that the RHS is further dominated by $(\min(i_1, i_2) + 1)(\min(j_1, j_2) + 1)$ multiplied with
\begin{align*}
 \calR\Big(\Big\{ \frac{|K|}{|I_1|^{1/2}|I_2|^{1/2}}\frac{|V|}{|J_1|^{1/2}|J_2|^{1/2}} & \pi_{\calD^k, b_{K, V, I_1,I_2, J_1, J_2, s}} \\
 & \in \calL(L^r(\R^k; L^q(\R^n \times Y^n; L^p(\R^m \times Y^m;E))))\colon \\
 & s \in \calS, K, I_1, I_2 \in \calD^n, V, J_1, J_2 \in \calD^m\Big\}\Big).
\end{align*}
This constant is bounded by Proposition \ref{prop:1parprodLattice}.

We can run the argument by decomposing in the symmetric way so that the partial paraproducts are formed in $\R^{n+k}$. This gives
that we can change $\R^n$ and $\R^m$ above, which yields all the symmetries.
\end{proof}

\end{document}